\documentclass[]{article} 
\usepackage[utf8]{inputenc}
\usepackage[english]{babel}
\usepackage{color,soul}
\usepackage{graphicx}
\usepackage{graphics}
\usepackage{amssymb,amsmath,amsthm,epsfig}
\newcommand{\norm}[1]{{\left\|{#1}\right\|}}

\newcommand{\scal}[1]{{\left\langle{#1}\right\rangle}}

\newcommand{\vp}{\psi^{(\alpha)}_{n,c}}

\newtheorem{theorem}{Theorem}
\newtheorem{lemma}{Lemma}

\newtheorem{proposition}{Proposition}

\newtheorem{remark}{Remark}
\newcounter{reh}
\newcounter{rek}
\vspace{5mm}
\numberwithin{equation}{section}
\begin{document}
	\begin{center}
		{\large {\bf Discrete Hankel Prolate Spheroidal Wave Functions: Spectral Analysis and Application. }}\\
		\vskip 1cm Mourad Boulsane$^*$ {\footnote{
				Corresponding author: Mourad Boulsane, Email: boulsane.mourad@hotmail.fr},\footnote{This work was supported by the tunisian DGRST Research Grant LR21ES10.}}
	\end{center}
	\vskip 0.5cm {\small
		\noindent $^*$ Carthage University,
		Faculty of Sciences of Bizerte, Department of  Mathematics Research Laboratory GAMA, Jarzouna, 7021, Tunisia.
	}
\begin{abstract}
Since the early 1960s, the fields of signal processing, data transmission, channel equalisation, filter design and others have been technologically developed and modernised as a result of the research carried out by D. Slepian and his co-authors H. J Landau and H. O Pollack on the time and band-limited wave system known as discrete and continuous spheroidal waves systems. Our aim in this paper is to introduce new discrete wave sequences called discrete Hankel Prolate spheroidal sequences {\bf DHPSS} and their counterparts in the frequency domain called discrete Hankel Prolate spheroidal wave functions {\bf DHPSWF} as radial parts of different solutions of a discrete multidimensional energy maximization problem similar to the one given by D. Slepian and which will generalize his classical pioneering work. In the meantime, we will ensure that our new family is the eigenfunctions set of a finite rank integral operator defined on $L^2(0,\omega),\,0<\omega<1,$ with an associated kernel given by $\sum_{k=1}^N\phi^{\alpha}_{n}(r)\phi^{\alpha}_{n}(r'),$ where $\phi^{\alpha}_{n}(r)=\frac{\sqrt{2r}J_{\alpha}(s_n^{(\alpha)}r)}{|J_{\alpha+1}(s^{(\alpha)}_n)|},0\leq r\leq 1.$ Here $J_\alpha$ is the Bessel function of the first kind and $(s_n^{(\alpha)})_n$ are the associated positive zeros.
In addition, we will extend the various classical results proposed concerning the decay rate and spectral distribution associated with the classical case, then we will finish our work by an application on the Ingham's universal constant which we will specify with an upper bound estimate. 
\end{abstract}
 Discrete prolate spheroidal wave functions and
sequences. Hankel transform.\\
{\bf Mathematics Subject Classification Primary:} 42A38 · 15B52. Secondary 60F10 · 60B20
\section{Introduction}
About half a century ago and more, D. Slepian, H. J Landau and H. O Pollack \cite{Slepian1}, \cite{L. P}, \cite{Slepian2} discovered the prolate spheroidal wave functions {\bf PSWFs}, they are the best essentially time and band-limited signals with bandwidth $c > 0$. On the others words, they are the keys of energy maximization problem defined on the classical Paley-Wiener space $\mathcal{B}_c$, the space of functions from
$L^2(\mathbb{R})$ with Fourier transforms supported on $[-c,c],$ that is 
 $$\mbox{ Find }f\in \mathcal{B}_c \mbox{ such that } f=\arg\max_{f\in\mathcal{B}_c}\frac{\norm{f}^2_{L^2(-1,1)}}{\norm{f}^2_{L^2(\mathbb{R})}}.$$
 Then, one of the
 most important properties is that the first {\bf PSWF} of bandlimited orthogonal set {\bf PSWFs} comprise the highest energy concentration within a limited frequency band and all succeeding {\bf PSWFs} have a decreasing concentration. Technically, from the signal reconstruction principle (inversion formula) and Parseval identity, {\bf PSWFs} become the eigenfunctions of the finite Fourier transform $\mathcal{F}_c$ defined on $L^2(-1,1)$ with associated eigenvalues $\mu$, that is
 \begin{equation}
 \mathcal{F}_c(\varphi)(x):=\int_{-1}^1\varphi(t)e^{-itx}dt=\mu\varphi(x),\,\,x\in(-1,1).
 \end{equation}
 Further, they obey a bi-orthogonality law respectively given on $L^2(-1,1)$ and $L^2(\mathbb{R})$. In particular, they form an orthonormal basis of $L^2(-1,1)$ and an orthogonal basis of the Paley-Wiener space $\mathcal{B}_c$. From a surprise incident, D. Slepian has proved that our integral operator $\mathcal{F}_c$ commutes with a Sturm-Liouville
 differential operator $\mathcal{L}_c$
 defined on $C^2(-1,1)$ by
 \begin{equation}
 \mathcal{L}_c(\varphi)=-\frac{d}{dx}[(1-x^2)\varphi']+c^2x^2\varphi.
 \end{equation} 
 Honestly, from the previous two heart points and Sturm Liouville theory with min-max theorem, literature began to be written about the different properties of the prolate spheroidal wave functions. In addition, in a series of papers given by \cite{B.K},\cite{B.J.K},\cite{H}, authors have described with more precision the local estimates of {\bf PSWFs} with the decay rate of the associated eigenvalues as well as these important different applications in many areas such that signal processing theory and recently analysis statistical. To be fair, we should also talk about the multidimensional prolate spheroidal waves functions and these radial part discovered by D. Slepian \cite{Slepian3} and more described in section $3$ and don't forget the real weighted case analysed by Wang in a series of papers given by \cite{W}, \cite{W.Z}. Nonetheless, those continuous wave functions encountered numerical computation problems for some frequency band. So, it has been proved that we can perfectly implement a continuous time and band limited signal by a discrete time sequence system, when the sampling frequency is chosen according to the sampling theorem, for more details we refer the reader to \cite{T.R.J}. Discrete prolate spheroidal wave sequences {\bf DPSSs} and their counterpart in frequency domain called discrete prolate spheroidal wave functions {\bf DPSWFs} introduced by D. Slepian \cite{Slepian0} in the 70s are good solutions of the previous problem. The {\bf DPSWFs}, denoted by $\left(\varphi^{N}_{n,\omega}\right)_{0\leq n\leq N-1},$ are
 characterized as the amplitude spectra (Fourier series) of index-limited complex sequences
 with index support $[[0,..., N-1]]$, that are most concentrated in the interval $(-\omega,\omega),\,0<\omega<1/2$. These {\bf DPSSs} are infinite sequences in $\ell^2(\mathbb{C})$ with amplitude spectra supported in
 $[-\omega,\omega]$, with coefficients mostly concentrated in the index range $[[0, . . . , N-1]]$. The most important property of {\bf DPSSs} is that their truncated versions over $[|0,N-1|]$ are the $N$ eigenvectors of the Toeplitz matrix $\tilde{\rho}_{N,\mathcal{K}}$ defined as follow
 \begin{equation}
 \tilde{\rho}_{N,\omega}=\left(
 \frac{\sin(2\pi(n-m)\omega)}{\pi(n-m)}
 \right)_{0\leq n,m\leq N-1},
 \end{equation} with same eigenvalues $\tilde{\lambda}_{n,N}(\omega)$ of the integral operator $\tilde{\mathcal{Q}}_{N,\omega}$ given in \eqref{e11} and to be described in more details later. The body of work has varied in this regard. From a purely mathematical perspective, the author in \cite{Slepian0} relied on two important properties of {\bf DPSWFs}. Firstly, they are the eigenfunctions of the finite rank integral transform $\tilde{\mathcal{Q}}_{N,\omega}$ defined on $L^2(-\omega,\omega)$ with associated degenerated kernel $\tilde{\mathcal{K}}_{N,\omega}(x,y)=\displaystyle\frac{\sin(N\pi(x-y))}{\sin(\pi(x-y))},$ that is
 \begin{equation}\label{e11}
 \tilde{\mathcal{Q}}_{N,\omega}(\varphi^{N}_{n,\omega})(x):=\int_{-\omega}^{\omega}\frac{\sin(N\pi(x-y))}{\sin(\pi(x-y))}\varphi^{N}_{n,\omega}(y)dy=\tilde{\lambda}_{n,N}(\omega)\varphi^{N}_{n,\omega}(x),\,\,x\in(-1/2,1/2).
 \end{equation}
 Here $\tilde{\lambda}_{n,N}(\omega)$ are the decreasing associated eigenvalues, for more informations about decay rate, we refer the reader to \cite{B.B.K}. Secondly, from a commutativity reason between $\tilde{\mathcal{Q}}_{N,\omega}$ and a Sturm Liouville differential operator denoted by $\tilde{\mathcal{L}}_{N,\omega}$, the {\bf DPSWFs} become these eigenfunctions with increasing eigenvalues $\tilde{\theta}_{n,N}(\omega)$, see \cite{Slepian0}, that is for every $0\leq n\leq N-1,$ we have
 \begin{equation}
 \tilde{\mathcal{L}}_{N,\omega}(\varphi^{N}_{n,\omega}):=\frac{d}{df}\left[\left(\cos(f)-\cos(2\pi\omega)\right)\frac{d\varphi^{N}_{n,\omega}}{df}\right]+\left[\frac{1}{4}(N^2-1)\cos(f)\right]=\tilde{\theta}_{n,N}(\omega)\varphi^{N}_{n,\omega},
 \end{equation}
 where $f=2\pi x,\,\,x\in(-1/2,1/2).$ In this paper, we are interested to formalize a new generation of discrete time sequences and their counterpart in the frequency domain (Fourier Bessel series) developed on the orthonormal basis of $L^2(0,1)$ and given by $\phi^{\alpha}_{n}(r)=\frac{\sqrt{2r}J_{\alpha}(s_n^{(\alpha)}r)}{|J_{\alpha+1}(s^{(\alpha)}_n)|},\alpha\geq-1/2$ and $0\leq r\leq 1.$ Here $J_\alpha$ is the Bessel function of the first kind and $(s_n^{(\alpha)})_n$ is the associated sequence of positive zeros. Our aim here is to generalize the classical case which that the previous family covers the {\bf DPSWFs} for $\alpha=\pm1/2.$ Let's be more specific here, for $\mathcal{K}=\mathbb{B}(0_{\mathbb{R}^d},\omega)\subset \mathbb{B}^d,\,0<\omega<1,$ $d\geq 2$ and $\mathbb{B}^d$ is the unit ball of $\mathbb{R}^d$, in treating the discrete energy maximization problem in many variables based on the orthonormal eigenfunctions basis $(\Phi_k)_k$ of the resolvent associated with the Laplacian $\Delta$ defined on $L^2(\mathbb{B}^d)$, as the most concentrated frequency function in the compact region $\mathcal{K}$ with index support $[[0,..., N-1]],$ that is
 \begin{equation}
\mbox{Find } x\in S_N \mbox{ such that } x= \arg\max_{x\in S_N} \frac{\|\hat{x}\|^2_{L^2(\mathcal{K})}}{\|\hat{x}\|^2_{L^2(\mathbb{B}^d)}}, 
 \end{equation}
  where $S_N=\{x\in\ell^2(\mathbb{C}), x_n=0,\,\forall n\geq N\}$ and $\hat{x}=\sum_{k=0}^{N-1}x_k\Phi_k,$ we are lead to the following integral equation, 
 \begin{equation}\label{12}
 \int_0^{\omega}\left(\sum_{k=1}^N\phi^{\alpha}_{n}(r)\phi^{\alpha}_{n}(r')\right)\varphi(r)dr=\lambda\varphi(r'),\,\,r'\in(0,1),
 \end{equation}
 more details is given in section $3.$ In other words, the radial part of the multidimensional discrete frequency prolate are the eigenfunctions of the integral operator $\tilde{\mathcal{Q}}^{\alpha}_{N,\omega}$ defined on $L^2(0,1)$ with associated kernel  given in \eqref{12} by $\tilde{K}^{\alpha}_{N,\omega}(r,r')=\chi_{(0,\omega)}(r')\sum_{k=1}^N\phi^{\alpha}_{n}(r)\phi^{\alpha}_{n}(r')$. They are the most concentrated discrete frequency (Fourier Bessel series) time in $(0,\omega)$ with index support $[|0,...,N-1|]$ called discrete Hankel prolate spheroidal wave functions {\bf DHPSWFs} denoted by $\left(\varphi_{n,\omega}^{(N,\alpha)}\right)_{0\leq n\leq N-1}.$ Its associated discrete time sequence {\bf DHPSSs} denoted by $\left(x^N_n\right)_{0\leq n\leq N-1}$ are the infinite sequences of $\ell^2(\mathbb{C})$  with amplitude spectra (Fourier Bessel series) supported in
 $(0,\omega)$, with coefficients mostly concentrated in the index range $[[0, . . . , N-1]].$ Moreover, from computational reasons inspired by the similar result given in the classical framework, the truncated of {\bf DHPSSs} over $[|0,...,N-1|]$ are the $N$ eigenvectors of the matrix
 \begin{equation*}
 \rho_{N,\omega}^{\alpha}=\begin{pmatrix}
 \frac{2K^{\alpha}_{\omega}(s_j^{(\alpha)},s_k^{(\alpha)})}{\sqrt{s_j^{(\alpha)}}|J_{\alpha+1}(s_j^{(\alpha)})|\sqrt{s_k^{(\alpha)}}|J_{\alpha+1}(s_k^{(\alpha)})|}
 \end{pmatrix}_{1\leq j,k\leq N},
 \end{equation*}
 where $K^{\alpha}_{\omega}$ is the continuous kernel case associated to the Hankel prolate spheroidal wave functions {\bf HPSWFs} given by $ K^{\alpha}_{\omega}(x,y)=\omega\, G_\alpha(\omega x,\omega y)$ and  
 \begin{equation}\label{e17}
 G_\alpha(x,y)=\begin{cases}\frac{\sqrt{xy}}{x^2-y^2}\left(xJ_{\alpha+1}(x)J_{\alpha}(y)-yJ_{\alpha+1}(y)J_{\alpha}(x)\right)& x\neq y\\ \frac{1}{2}\left(\left(xJ_{\alpha+1}(x)\right)'J_{\alpha}(x)-xJ_{\alpha+1}(x)J_{\alpha}'(x)\right)&  x=y\end{cases}.
 \end{equation}
   Our objective from this work
 is to push forward the different results given in the classical case $(\alpha=\pm1/2)$ from \cite{B.B.K} and provide the reader with some useful explicit non asymptotic results about the decay rate of $\tilde{\mathcal{Q}}^{\alpha}_{N,\omega}-$spectrum denoted by  $\tilde{\lambda}^{N}_{n,\alpha}(\omega).$ That is for every $0\leq n\leq N-1,$ we have
 \begin{equation}
 \tilde{\lambda}_{n,\alpha}^N(c)\lesssim\frac{1}{\zeta}\frac{\sqrt{2}}{\sqrt{e(2N+\alpha+3/2)}}\left(\frac{e\omega s_{N+1}^{(\alpha)}}{4n+2\alpha+3}\right)^{2n+\alpha+1/2}.
 \end{equation}
 Here $\zeta=\frac{1}{2}\displaystyle\min_{1\leq i\leq N}\left\{\left|s_{i+1}^{(\alpha)}-s_i^{(\alpha)}\right|\right\}.$
 Further, we take advantage from the comparison property between the two spectra of the discrete and continuous cases when we will ensure with more precision the decay rate of the eigenvalues $\tilde{\lambda}_{n,\alpha}^N(c)$ from the different results about the non asymptotic decay rate of $\lambda_n^{\alpha}(c)$ the eigenvalue associated with the continuous Hankel prolate given in \cite{M.B}, that is for $c=\omega (s_N^{(\alpha)}+\varepsilon)$ and $0\leq n\leq N-1,$ we have
 \begin{equation*}
 \tilde{\lambda}_{n,\alpha}^N(\omega)\sim \lambda_{n}^{\alpha}\left(c\right).
 \end{equation*}
 Here $(s_N^{(\alpha)})_N$ are the positive zeros of Bessel function $J_\alpha$ of the first kind and $\varepsilon>0$ is a real constant to be fixed later. Moreover, we will give an approximation of the kernel $\tilde{K}^{\alpha}_{N,\omega}$, which is the key for more results about the distribution of  $\tilde{\mathcal{Q}}^{\alpha}_{N,\omega}-$spectrum by using estimates of $\mbox{Trace}(\tilde{\mathcal{Q}}^{\alpha}_{N,\omega})$ and his Hilbert-Schmidt norm $\norm{\tilde{\mathcal{Q}}^{\alpha}_{N,\omega}}_{HS}$, that is 
 \begin{equation*}
 \tilde{K}^{(\alpha)}_{N,\omega}(x,y)=K_{c_N}^{\alpha}(x,y)+F_{c_N}^{\alpha}(x,y).
 \end{equation*}
 Here $K_{c_N}^{\alpha}$ is the kernel associated with the operator $\mathcal{Q}_c^{\alpha}$ given by \eqref{e17} and for $V(r)=\int_0^{\infty}\frac{\sinh(rt)}{1+e^{-2t}}e^{-2t}dt,$ we have $$F_{c_N}^{\alpha}(x,y)\sim\frac{2}{\pi}\left[\sin((x+y)c_N-2\gamma_\alpha)V(x+y)+\sin((x-y)c_N)V(x-y)\right].$$
 More precisely, based on the previous estimate of the kernel $\tilde{K}^{(\alpha)}_{N,\omega}$ using the continuous kernel $K_{c_N}^{\alpha}$, we obtain the following result.
  For $0<\omega<1$, $\alpha\geq-1/2$ be two real numbers and $N\geq N_0$ be an integer, then for $c=(N+\frac{1}{2}\alpha+\frac{1}{4})\pi\omega$ and $0<\varepsilon<1/2,$  we have
  \begin{equation*}
  \displaystyle\#\{n,\,\varepsilon<\tilde{\lambda}_{n,\alpha}^N(\omega)<1-\varepsilon\}\leq\frac{M_{\alpha}}{\varepsilon(1-\varepsilon)}\left[\frac{\omega^2}{c}+\ln(c)+G(\omega)\right]
  \end{equation*}
 where $G(w)$ is specific given by proposition $4$ of section $4$ and $M_\alpha$ is a constant depending only of $\alpha.$
 
 Finally, this work is organized as follows. In section $2$, we will formalize the generalized discrete time and their counterpart in frequency domain over the weighted Hilbert space $L^2(\mathcal{O},\omega)$, where $\mathcal{O}$ is a subset of $\mathbb{R}^d,$ using an orthonormal basis denoted $(\psi_n)_n$ instead of the trigonometric exponential basis given in the classical case. In section $3$, we will apply the general principle to the eigenfunctions basis associated with the resolvent of Laplacian operator $\Delta$ over $L^2(\mathbb{B}^d),\,d\geq 2$ where $\mathbb{B}^d$ is the unit multidimensional ball. In treating the Generalized Discrete Prolate Spheroidal Wave Functions as the most concentrated frequency function in the compact region $\mathcal{K}=\mathbb{B}(0_{\mathbb{R}^d},\omega)\subset \mathbb{B}^d,\,0<\omega<1,$ with index support $[[0,..., N-1]],$ we are lead to prove that the radial part are the most concentrated discrete frequency time functions in $(0,\omega)$ with index support $[|0,...,N-1|]$ called discrete Hankel prolate spheroidal wave functions {\bf DHPSWFs}. In the meantime, we will describe the associated discrete Hankel prolate sequence {\bf DHPSSs}. Then, we will finished by recall some properties of the continuous case called Hankel prolate {\bf HPSWFs} discovered by D. Slepian in the $60$s. Taking into account that the {\bf DHPSWFs} are the eigenfunctions of a compact integral operator, we will focus in Section $4$ on the fast decay rate of the associated eigenvalues based on the Min-Max characterization. Moreover, we study some interesting connections between the eigenvalues associated with {\bf DHPSWFs} and their
 corresponding continuous case associated with  {\bf HPSWFs}. Based on these connections, we deduce
 various results on their distribution and decay rates. Finally, as an application, we are dealing with Ingham's universal constant which we will specify with a more precise upper estimate.    
\section{The general principle}
Let $N\geq 1$ be an integer and $\mathcal{O}$ is a subset of the euclidean space $\mathbb{R}^d,\,d\geq 1$. Let fixed an orthonormal basis of the weighted Hilbert space $L^2(\mathcal{O},\omega)$ denoted by $(\psi_n)_n$ and a compact subset $\mathcal{K}$ of $\mathcal{O}$, then we define the Paley-Wiener space $\mathcal{B}_{\mathcal{K}}$ the space of band-limited sequence. More precisely, the later is the subspace of $\ell^2(\mathbb{C})$ with supported Fourier series in $\mathcal{K},$ that is
\begin{equation}
\mathcal{B}_{\mathcal{K}}=\{x\in \ell^2(\mathbb{C}),\,\mbox{ Support }(\hat{x})\subset\mathcal{K} \}.
\end{equation}
where $\hat{x}=\sum_{n=0}^{+\infty}x_n\psi_n$. Inspired from the classic framework given by D. Slepian in \cite{Slepian0}, we will going to reform the generalized discrete prolate spheroidal wave sequence {\bf GDPSSs} as the sequence with band limited amplitude spectra $\hat{x}$ and most concentrated in index over $[|0,...,N-1|]$ from the following energy maximization problem:
\begin{equation}\label{e0}
	\mbox{Find } x\in \mathcal{B}_{\mathcal{K}} \mbox{ such that } x= \arg\max_{x\in \mathcal{B}_{\mathcal{K}}} \frac{\|x\|^2_{\ell^2(0,N-1)}}{\|x\|^2_{\ell^2(\mathbb{C})}}. 
\end{equation}
Indeed, let consider a sequence $x\in\mathcal{B}_{\mathcal{K}}$, then we have
\begin{eqnarray*}
\norm{x}^2_{\ell^2(0,N-1)}&=&\sum_{n=0}^{N-1}|x_n|^2=\sum_{n=0}^{N-1}x_n\overline{x_n}\\&=&\sum_{n=0}^{N-1}\int_{\mathcal{O}}\hat{x}(t)\overline{\psi_n}(t)\omega(t)dt\int_{\mathcal{O}}\overline{\hat{x}(s)}\psi_n(s)\omega(s)ds\\&=&\int_{\mathcal{K}}\overline{\hat{x}(s)}\left(\int_{\mathcal{K}}\hat{x}(t)\sum_{n=0}^{N-1}\psi_n(s)\overline{\psi_n}(t)\omega(t)dt\right)\omega(s)ds.
\end{eqnarray*}
From Parseval identity and since $x\in\mathcal{B}_{\mathcal{K}}$, we have $\norm{x}^2_{\ell^2(\mathbb{C})}=\norm{\hat{x}}^2_{L^2(\mathcal{O},\omega)}=\displaystyle\int_{\mathcal{K}}\overline{\hat{x}(s)}\hat{x}(s)\omega(s)ds.$ Consequently, we can easily check that the maximum given in \eqref{e0} is the largest eigenvalue of the integral equation
\begin{equation}\label{1}
\int_{\mathcal{\mathcal{K}}}K(s,t)\varphi (t)\omega(t)dt=\lambda \varphi(s),\,\,s\in \mathcal{K},	
\end{equation}
where $K(s,t)=\displaystyle\sum_{n=0}^{N-1}\psi_n(s)\overline{\psi_n}(t),\,\,(s,t)\in \mathcal{K}^2$ and $\varphi(s)=\displaystyle\sum_{n=0}^{+\infty}x_n\psi_n(s)$ with $x_n=\displaystyle\int_{\mathcal{K}}\varphi(t)\psi_n(t)\omega(t)dt.$\\
Let $S_N$ be the subspace of $\ell^2(\mathbb{C})$ with index limited sequence over $[|0,...,N-1|]$, that is 
$$S_N=\{x\in\ell^2(\mathbb{C}),\,x_n=0\,\,\forall n\geq N\}.$$
From an equivalent energy maximization problem introduced also by D. Slepian, given by 
\begin{equation}\label{e1}
\mbox{Find } x\in S_N \mbox{ such that } x= \arg\max_{x\in S_N} \frac{\|\hat{x}\|^2_{L^2(\mathcal{K},\omega)}}{\|\hat{x}\|^2_{L^2(\mathcal{O},\omega)}}, 
\end{equation}
we can easily check that $\frac{1}{\lambda}x\chi_{[|0,...,N-1|]}$, where $x$ is the {\bf GDPSS} solution of \eqref{e0}, is the index limited sequence over $[|0,...,N-1|]$ with amplitude spectra most concentrated on $\mathcal{K}$. Indeed, let $y\in S_N,$ since we have
\begin{eqnarray*}
\hat{y}(s)&=&\displaystyle\sum_{n=0}^{N-1}y_n\psi_n(s)=\sum_{n=0}^{N-1}\int_\mathcal{O}\hat{y}(t)\overline{\psi_n}(t)\omega(t)dt\psi_n(s)\\&=&\int_{\mathcal{O}}\hat{y}(t)\sum_{n=0}^{N-1}\overline{\psi_n}(t)\psi_n(s)\omega(t)dt,
\end{eqnarray*}
then, we get
\begin{eqnarray*}
\norm{\hat{y}}^2_{L^2(\mathcal{K},\omega)}&=&\int_{\mathcal{K}}\overline{\hat{y}(s)}\hat{y}(s)\omega(s)ds\\&=&\int_{\mathcal{K}}\hat{y}(s)\int_{\mathcal{O}}\overline{\hat{y}(t)}\sum_{n=0}^{N-1}\psi_n(t)\overline{\psi_n}(s)\omega(t)dt\omega(s)ds\\&=&\int_{\mathcal{O}}\overline{\hat{y}(t)}\int_{\mathcal{K}}\hat{y}(s)\sum_{n=0}^{N-1}\psi_n(t)\overline{\psi_n}(s)\omega(s)ds\omega(t)dt.
\end{eqnarray*}
Finally, from the fact that $\norm{\hat{y}}^2_{L^2(\mathcal{O},\omega)}=\displaystyle\int_{\mathcal{O}}\overline{\hat{y}(t)}\hat{y}(t)\omega(t)dt$, we can see that the largest eigenvalue of the integral equation expanded over $\mathcal{O}$ given in \eqref{1} is the solution of \eqref{e1}, that is
\begin{equation}\label{e.3}
\int_{\mathcal{\mathcal{K}}}K(s,t)\varphi (t)\omega(t)dt=\lambda \varphi(s),\,\,s\in \mathcal{O},	
\end{equation}
where $\varphi(s)=\sum_{n=0}^{N-1}y_n\psi_n(s)$ and from \eqref{e.3}, we have $\displaystyle\int_{\mathcal{K}}\varphi(t)\psi_n(t)\omega(t)dt=\lambda y_n,\,0\leq n\leq N-1.$
Consequently, we obtain an orthonormal system of eigenfunctions denoted by $(\varphi^N_n)_{0\leq n\leq N-1}$ called Generalized Discrete prolate spheroidal wave functions {\bf GDPSWFs} of the finite rank integral operator $\tilde{\mathcal{Q}}_N$ defined on $L^2(\mathcal{O},\omega)$ with degenerated kernel $\tilde{K}(s,t)=\chi_{\mathcal{K}}(t)\sum_{n=0}^{N-1}\psi_n(s)\overline{\psi_n}(t),\,(s,t)\in \mathcal{O}^2$ and associated eigenvalues $(\tilde{\lambda}^N_n)_{0\leq n\leq N-1}$ arranged in the decreasing order $1>\tilde{\lambda}^N_0>\tilde{\lambda}^N_1>...>\tilde{\lambda}^N_{N-1}>0$, that is for every $0\leq n\leq N-1$, we have
\begin{equation}\label{e.1}
\tilde{\mathcal{Q}}_N(\varphi^N_n)(s)=\int_{\mathcal{K}}\tilde{K}(s,t)\varphi^N_n(t)\omega(t)dt=\tilde{\lambda}^N_n\varphi^N_n(s),\,\,s\in \mathcal{O}.
\end{equation}
It is very important to give the relationship between {\bf GDPSWFs} denoted by $(\varphi^N_n)_{0\leq n\leq N-1}$ and {\bf GDPSSs} denoted by $(x^N_{n})_{0\leq n\leq N-1}$, that is for every $0\leq n\leq N-1,$ we have
\begin{equation}\label{e5}
 \varphi^N_n(s)=\begin{cases}\displaystyle\sum_{k=0}^{N-1}x_{k,n}^N\psi_k(s),\,\,\, s\in \mathcal{O}\\\displaystyle\tilde{\lambda}^N_n\sum_{k=0}^{+\infty}x_{k,n}^N\psi_k(s),\,\,\, s\in \mathcal{K}\end{cases}.
 \end{equation}
We should point out here that {\bf GDPSWFs}  satisfy an interesting bi-orthogonality property given by the following identity
\begin{equation}\label{e4}
\int_{\mathcal{K}}\varphi^N_n(t)\overline{\varphi^N_m}(t)\omega(t)dt=\tilde{\lambda}^N_n\int_{\mathcal{O}}\varphi^N_n(t)\overline{\varphi^N_m}(t)\omega(t)dt=\tilde{\lambda}^N_n\delta_{n,m}.
\end{equation}
Indeed, from the self adjoint property of $\tilde{\mathcal{Q}}_N$, $\left(\varphi^N_n\right)$ form an orthonormal system of $L^2(\mathcal{O},\omega)$ and from the fact that $\int_{\mathcal{K}}\varphi^N_n(s)\psi_k(s)\omega(s)ds=\tilde{\lambda}^N_nx_{k,n}^N$ given by \eqref{e.1}, we obtain the desired identity \eqref{e4}.
We should mention here that we gratefully acknowledge to the previous identity, especially over some particular properties of {\bf GDPSSs}. Indeed, from \eqref{e4}, the discrete sequence family satisfy a bi-orthogonal relation, that is
\begin{equation}
\displaystyle\sum_{k=0}^{N-1}x_{k,n}^Nx_{k,m}^N=\tilde{\lambda}^{N}_{n}\sum_{k=0}^{\infty}x_{k,n}^Nx_{k,m}^N=\delta_{n,m},\,\,0\leq n,m\leq N-1.
\end{equation}
Further more, the associated truncated over $[|0,...,N-1|]$ of {\bf GDPSSs} are the $N$ eigenvectors of the matrix $\tilde{\rho}_{N,\mathcal{K}}$ defined as follow
\begin{equation}
\tilde{\rho}_{N,\mathcal{K}}=\left(
\int_{\mathcal{K}}\psi_i(s)\psi_j(s)ds
\right)_{0\leq i,j\leq N-1}.
\end{equation}
Indeed, as we have  $\displaystyle\int_{\mathcal{K}}\varphi^N_n(s)\psi_k(s)\omega(s)ds=\tilde{\lambda}^N_nx_{k,n}^N$ and from \eqref{e5}, we can easily check that for every $0\leq n\leq N-1$, we have
\begin{equation}
\sum_{j=0}^{N-1}\tilde{\rho}_{N,\mathcal{K}}(i,j)x_{j,n}^{N}=\lambda_n^Nx_{i,n}^N,\,\,0\leq i\leq N-1.
\end{equation}

\section{Hankel Prolate Spheroidal wave functions: Discrete and continuous case}
\subsection{ Hankel Prolate spheroidal wave functions: Discrete case}
Let $d\geq 2$, we denoted by $\Delta$ the Laplacian defined on $\mathbb{B}^d$ the unit ball of the euclidean space $\mathbb{R}^d.$ It's well knowing that the resolvent of $\Delta$ is a compact operator, hence there exist an orthonormal eigenfunctions basis of $L^2(\mathbb{B}^d)$ given by the following
\begin{equation}\label{e.0}
\Phi_{n,m,k}^{d}(x)=\phi^{d}_{n,m}(\norm{x})Y_k^{(m)}\left(\frac{x}{\norm{x}}\right)\,\,\,n,m=1,2...\,,x\in\mathbb{B}^d,
\end{equation}
where $\left(Y_k^{(m)}\right)_{1\leq k\leq d(m)}$ is an orthonormal basis of the Spherical Harmonics of degree $m$ satisfying the following orthogonality relation 
\begin{equation}\label{e.4}
\int_{\mathbb{S}^{d-1}}Y_k^{(m)}(x)Y_j^{(n)}(x)d\sigma(x)=\delta_{k,j}\delta_{n,m}.
\end{equation} 
Here $\sigma$ is the surface of the unit sphere $\mathbb{S}^{d-1}$ and $d(m)=\frac{2m+d-2}{m}\begin{pmatrix}
m+d-3\\m-1
\end{pmatrix}$. For the rest of \eqref{e.0}, we consider the set of the positive zeros of Bessel function $J_{\alpha}$ of the first kind with order $\alpha\geq -1/2$ denoted by $(s_n^{(\alpha)})_{n\geq 1}$  satisfies the following estimate given in \cite{Elbert} by $s_n^{(\alpha)}=\pi n+\frac{\pi}{2}(\alpha-\frac{1}{2})+O(\frac{1}{n})$. For every $n,m\geq1,$ 
we define the normalized Bessel function $\phi^{d}_{n,m}$ on $(0,1)$ by  $$\phi^{d}_{n,m}(r)=r^{-\frac{d-2}{2}}\frac{\sqrt{2}J_{m+\frac{d-2}{2}}(s_n^{(m+\frac{d-2}{2})}r)}{|J_{m+\frac{d}{2}}(s_n^{(m+\frac{d-2}{2})})|}.$$ It's well known that for a fixed $m\geq 1,$ the previous system forms an orthonormal basis on $L^2\left((0,1),r^{d-1}dr\right)$, see \cite{H}, that is for every $\alpha\geq-1/2,$ we have
\begin{equation}
\int_0^1xJ_{\alpha}(s_i^{(\alpha)}x)J_{\alpha}(s_j^{(\alpha)}x)dx=\frac{|J_{\alpha+1}(s_j^{(\alpha)})|^2}{2}\delta_{i,j}.
\end{equation}
In order to define a new generation of discrete prolate, we will use the general principle with $\mathcal{O}=\mathbb{B}^d$ the unit ball and for a fixed $0<\omega<1$, $\mathcal{K}=\mathbb{B}^d(0_{\mathbb{R}^d},\omega)$ is the closure ball with radius $\omega$ and centre $0_{\mathbb{R}^d}.$ For ending, we will take the orthonormal basis of $L^2(\mathbb{B}^d)$ defined above by $\left(\Phi_{n,m,k}^d\right),$ then from the general principle section, we can easily check that the integral equation \eqref{1} have a bounded solution $f$ defined on $L^2(\mathbb{B}^d)$, that is for every $N,M\geq 1$, we have 
\begin{equation}\label{2}
\int_{\mathcal{K}}K^{(d)}_{N,M}(x,y)f(y)dy=\lambda f(x),\,\,x\in\mathbb{B}^d,
\end{equation}
where
\begin{equation}
 K^{(d)}_{N,M}(x,y)=\displaystyle\sum_{n,m=1}^{N,M}\phi^{d}_{n,m}(\norm{x})\phi^{d}_{n,m}(\norm{y})\sum_{k=1}^{d(m)}Y_k^{(m)}(\frac{x}{\norm{x}})Y_k^{(m)}(\frac{y}{\norm{y}}).\end{equation}
 For more details, let's write \eqref{2} on the other form
 \begin{equation}\label{3}
 \sum_{m=1}^{M}\sum_{k=1}^{d(m)}Y_k^{(m)}(\frac{x}{\norm{x}})\sum_{n=1}^NC^k_{n,m}(f)\phi_{n,m}^d(\norm{x})=\lambda f(x),\,\,x\in\mathbb{B}^d,
 \end{equation}
 where $C^k_{n,m}(f)=\displaystyle\int_{\mathcal{K}}\phi^{d}_{n,m}(\norm{y})Y_k^{(m)}(\frac{y}{\norm{y}})f(y)dy.$ It's clear that from the previous equation we have $f\in\mbox{ Span }\{\Phi_{n,m,}^d, 1\leq n,m\leq N,M\}$, then $f$ have an expansion series given by $$f(x)=\sum_{m=1}^{M}\sum_{k=1}^{d(m)}Y_k^{(m)}(\frac{x}{\norm{x}})\sum_{n=1}^N\delta^k_{n,m}(f)\phi_{n,m}^d(\norm{x}),\,\,x\in\mathbb{B}^d,$$
 where $\delta^k_{n,m}(f)=\displaystyle\int_{\mathcal{B}^d}\phi^{d}_{n,m}(\norm{y})Y_k^{(m)}(\frac{y}{\norm{y}})f(y)dy.$ Hence, we deduce that if we denoted by $G^N_{m,k}(r)=\sum_{n=1}^N\delta^k_{n,m}(f)\phi_{n,m}^d(r)$, we obtain the following equality
 \begin{equation}\label{4}
 f(x)=\sum_{m=1}^{M}\sum_{k=1}^{d(m)}G^N_{m,k}(\norm{x})Y_k^{(m)}(\frac{x}{\norm{x}}),\,\,x\in\mathbb{B}^d.
 \end{equation}
 Let $x=R\xi$ and $y=r\xi'$, $\xi, \xi'\in\mathbb{S}^{d-1},$ then we rewrite \eqref{2} under the following form
 \begin{equation}
 \int_0^{\omega}r^{d-1}\int_{\mathbb{S}^{d-1}}K^{(d)}_{N,M}(R\xi,r\xi')f(r\xi')d\sigma(\xi')dr=\lambda f(R\xi),\,\,R\in(0,1).
 \end{equation}
 Using \eqref{4} and the previous equation, we get
 \begin{equation}\label{5}
 \sum_{m=1}^{M}\sum_{k=1}^{d(m)}\int_0^{\omega}r^{d-1}G^N_{m,k}(r)\int_{\mathbb{S}^{d-1}}K^{(d)}_{N,M}(R\xi,r\xi')Y_k^{(m)}(\xi')d\sigma(\xi')dr=\lambda\sum_{m=1}^{M}\sum_{k=1}^{d(m)}G^N_{m,k}(R)Y_k^{(m)}(\xi).
 \end{equation}
 On the other hand, from \eqref{e.4}, we have
 \begin{eqnarray}\label{6}
 \int_{\mathbb{S}^{d-1}}K^{(d)}_{N,M}(R\xi,r\xi')Y_k^{(m)}(\xi')d\sigma(\xi')&=&\sum_{i,j=1}^{N,M}\phi^{d}_{i,j}(R)\phi^{d}_{i,j}(r)\sum_{\ell=1}^{d(j)}Y_{\ell}^{(j)}(\xi)\int_{\mathbb{S}^{d-1}}Y_{\ell}^{(j)}(\xi')Y_k^{(m)}(\xi')d\sigma(\xi')\nonumber\\&=&\sum_{i=1}^{N}\phi^{d}_{i,m}(R)\phi^{d}_{i,m}(r)Y_{k}^{(m)}(\xi)
 \end{eqnarray}
 Finally, if we replace \eqref{6} in \eqref{5}, we get the following result
 \begin{equation*}
 \sum_{m=1}^{M}\sum_{k=1}^{d(m)}\int_0^{\omega}r^{d-1}G^N_{m,k}(r)\sum_{i=1}^{N}\phi^{d}_{i,m}(R)\phi^{d}_{i,m}(r)drY_{k}^{(m)}(\xi)=\lambda\sum_{m=1}^{M}\sum_{k=1}^{d(m)}G^N_{m,k}(R)Y_k^{(m)}(\xi).
 \end{equation*}
 Consequently, if we identify the radial part from both sides we deduce that for every $1\leq m\leq M$ and $1\leq k\leq d(m)$, we have
 \begin{equation}\label{7}
 \int_0^{\omega}r^{d-1}G^N_{m,k}(r)\sum_{i=1}^{N}\phi^{d}_{i,m}(R)\phi^{d}_{i,m}(r)dr=\tilde{\lambda}^{N}_{m}(\omega)G^N_{m,k}(R),\,\,0\leq R\leq 1,
 \end{equation}
 from which it is seen that $G^N_{m,k}$ is independent of $k.$
 In the sequel, we consider the new orthonormal basis of $L^2(0,1)$ such that for every $\alpha=m+\frac{d-2}{2}\geq -1/2$, we have $$\phi^{(\alpha)}_{n}(r)=r^{\frac{d-1}{2}}\phi^{d}_{n,m}(r)=\frac{\sqrt{2r}J_{\alpha}(s_n^{(\alpha)}r)}{\left|J_{\alpha+1}(s_n^{(\alpha)})\right|},\,\,0\leq r\leq 1.$$
 They are the eigenfunctions of the positive self adjoint differential operator $(-\mathcal{L})$ defined on $C^2(0,1)$ by  $(-\mathcal{L})(V)=-V''-\left(\frac{1/4-\alpha^2}{x^2}\right)V$ with associated eigenvalues $\left((s_j^{(\alpha)})^2\right)_{j\geq1}$.
 Finally, we consider the finite rank self adjoint operator $\tilde{\mathcal{Q}}^{\alpha}_{N,\omega}$ defined on $L^2(0,1)$ with associated kernel given by $\tilde{K}^{(\alpha)}_{N,\omega}(x,y)=\sum_{i=1}^{N}\phi^{(\alpha)}_{i}(x)\phi^{(\alpha)}_{i}(y)$, then we set $\varphi_{n,N}^{(\alpha)}$ the associated eigenfunctions called discrete Hankel prolate spheroidal wave functions {\bf (DHPSWF)} and $\tilde{\lambda}^{N}_{n,\alpha}(\omega)$ the associated eigenvalues. Hence, from \eqref{7}, we have
 \begin{equation}\label{11}
 \int_0^{\omega}\tilde{K}^{(\alpha)}_{N,\omega}(x,y)\varphi_{n,N}^{(\alpha)}(y)dy=\tilde{\lambda}^{N}_{n,\alpha}(\omega)\varphi_{n,N}^{(\alpha)}(x),\,\,0\leq x\leq 1,\,\,0\leq n\leq N-1.
 \end{equation}
 We should mention here that from the previous details, the family $\varphi_{n,N}^{(\alpha)}$ form an orthonormal system of $L^2(0,1)$ and they satisfy the following bi-orthogonal relation
 \begin{equation}\label{8}
 \int_0^{\omega}\varphi_{n,N}^{(\alpha)}(y)\varphi_{m,N}^{(\alpha)}(y)dy=\tilde{\lambda}^{N}_{n,\alpha}(\omega)\int_0^{1}\varphi_{n,N}^{(\alpha)}(y)\varphi_{m,N}^{(\alpha)}(y)dy=\tilde{\lambda}^{N}_{n,\alpha}(\omega)\delta_{n,m},\,\,0\leq n,m\leq N-1.
 \end{equation}
 \subsection{Discrete Hankel Prolate spheroidal sequence}
From the general principle section, the discrete Hankel prolate spheroidal sequences {\bf DHPSSs} denoted by $\left(x^N_n\right)_{0\leq n\leq N-1}$ are solution of energy maximization problem given in \eqref{e0} over the Hankel Paley-Wiener space
 \begin{equation*}
 \mathcal{B}_{\omega}=\{x\in \ell^2(\mathbb{C}),\,\mbox{ Support }(\hat{x})\subset[0,\omega] \},
 \end{equation*}
 where $\hat{x}=\sum_{k=1}^{\infty}x_k\phi^{(\alpha)}_{k}.$ They are the best essentially index over $[|1,...,N|]$ and band-limited amplitude with bandwidth $\omega>0$ with associated index-restriction over $[|1,...,N|]$ are the best essentially frequency(Fourier Bessel series).
 The {\bf DHPSSs} have a connection with {\bf DPSWFs}, that is for every $0\leq n\leq N-1,$ we have
 \begin{equation}\label{9}
 \varphi_{n,N}^{(\alpha)}(x)=\displaystyle\sum_{k=1}^{N}x_{k,n}^N\phi^{(\alpha)}_{k}(x),\,\, x\in [0,1],
 \end{equation}
 where $x^N_{k,n}=\displaystyle\int_0^1\varphi_{n,N}^{(\alpha)}(x)\phi^{(\alpha)}_{k}(x)dx.$ Furthermore, by replacing \eqref{9} in the right side of \eqref{11} and we identify the two expansion coefficients, we get
 \begin{equation}\label{10} \int_0^{\omega}\varphi_{n,N}^{(\alpha)}(y)\phi^{(\alpha)}_{k}(y)dy=\tilde{\lambda}^{N}_{n,\alpha}(\omega)x_{k,n}^N,\,1\leq k\leq N.\end{equation} Hence, if we replace \eqref{9} in \eqref{10}, we can easily check that the truncated of each {\bf DHPSS} over $[|1,...,N|]$ are solution of the following system. For every $0\leq n\leq N-1$, we have
 \begin{equation}
 \sum_{j=1}^{N}\frac{2K^{\alpha}_{\omega}(s_j^{(\alpha)},s_k^{(\alpha)})}{\sqrt{s_j^{(\alpha)}}|J_{\alpha+1}(s_j^{(\alpha)})|\sqrt{s_k^{(\alpha)}}|J_{\alpha+1}(s_k^{(\alpha)})|}x_{j,n}^N=\tilde{\lambda}^{N}_{n,\alpha}(\omega)x_{k,n}^N,\,\,1\leq k\leq N,
 \end{equation}
 where $K^{\alpha}_{\omega}$ is the kernel of the integral operator $\mathcal{Q}_{\omega}^{\alpha}$ associated with the continuous case given in the next section \eqref{13}. Consequently, we consider the $N$ vectors obtained by truncating each {\bf DHPSS} to the index set $[|1,...,N|]$ denoted by $\left(V_n^{N}\right)_{0\leq n\leq N-1}$  such that $V_n^{N}=\left(x_{1,n}^N,...,x_{N,n}^N\right)^T$, then they are the $N$ eigenvectors of the Matrix
 \begin{equation}
 \rho_{N,\omega}^{\alpha}=\begin{pmatrix}
 \frac{2K^{\alpha}_{\omega}(s_j^{(\alpha)},s_k^{(\alpha)})}{\sqrt{s_j^{(\alpha)}}|J_{\alpha+1}(s_j^{(\alpha)})|\sqrt{s_k^{(\alpha)}}|J_{\alpha+1}(s_k^{(\alpha)})|}
 \end{pmatrix}_{1\leq j,k\leq N}
 \end{equation} 
 Finally,  we should mention that from \eqref{8} and \eqref{9}, {\bf DHPSS} obey an orthogonality law given by the following rule
 \begin{equation}
 \displaystyle\sum_{k=1}^{N}x_{k,n}^Nx_{k,m}^N=\tilde{\lambda}^{N}_{n,\alpha}(\omega)\sum_{k=1}^{\infty}x_{k,n}^Nx_{k,m}^N=\delta_{n,m}.
 \end{equation}
 \subsection{Hankel Prolate Spheroidal wave functions: Continuous case}
 We recall that for a given values of $\alpha\geq-\frac{1}{2} $ and $c>0$, Hankel prolate spheroidal wave functions called also circular prolate (HPSWFs), denoted by $\varphi^{\alpha}_{n,c}$, have been discovered by an interesting paper \cite{Slepian3} given by D. Slepian. They are defined as the radial part of the classical multidimensional prolate defined on $L^2(\mathbb{B}^d),\,d=p+2,\,p\in\mathbb{N}$, where $\mathbb{B}^d$ is the unit disk of $\mathbb{R}^d.$ In the meantime, author have been showed that HPSWFs are the best essentially time and band-limited amplitude signals with bandwidth $c>0$ defined on $L^2(0,+\infty)$ with these restriction on $(0,1)$ are the most concentrated in frequency. Furthermore, HPSWFs are the different band-limited eigenfunctions of the finite Hankel transform $\mathcal{H}_c^{\alpha}$ defined on $L^2(0,1)$ with kernel $H_c^{\alpha}(x,y)=\sqrt{cxy}J_{\alpha}(cxy)$, here $J_{\alpha}$ is the Bessel function of the first type and order $\alpha>-\frac{1}{2}$, that is 
 \begin{equation}
 \mathcal{H}_c^{\alpha}(\varphi^{\alpha}_{n,c})=\mu_{n,\alpha}(c)\varphi^{\alpha}_{n,c}.
 \end{equation}
 To the operator $\mathcal{H}_c^{\alpha},$ we associate a positive, self-adjoint compact integral operator $\mathcal{Q}_c^{\alpha}=c\mathcal{H}_c^{\alpha}\mathcal{H}_c^{\alpha}$ defined on $L^2(0,1)$  with kernel $K_c^{\alpha}(x,y)=c\,G_{\alpha}(cx,cy),$ where
 \begin{equation}\label{13} 
 G_{\alpha}(x,y)=\begin{cases}\frac{\sqrt{xy}}{x^2-y^2}\left(xJ_{\alpha+1}(x)J_{\alpha}(y)-yJ_{\alpha+1}(y)J_{\alpha}(x)\right)& x\neq y\\ \frac{1}{2}\left(\left(xJ_{\alpha+1}(x)\right)'J_{\alpha}(x)-xJ_{\alpha+1}(x)J_{\alpha}'(x)\right)&  x=y\end{cases}
 \end{equation}
 We denote by $\lambda_{n}^{\alpha}(c)$ the infinite and countable sequence of the eigenvalues of the operator $\mathcal{Q}_c^{\alpha}$, that is
 $\lambda_{n}^{\alpha}(c)=c|\mu_{n}^{\alpha}(c)|^2.$
 In his pioneer work \cite{Slepian3}, D. Slepian has shown that the compact integral operator $\mathcal{H}_c^{\alpha}$ commutes with the following differential operator $\mathcal D^{\alpha}_c$ defined on $C^2([0,1])$ by
 
 \begin{equation}\label{differ_operator1}
 \mathcal D^{\alpha}_c (\phi)(x) = \dfrac{d}{dx} \left[ (1-x^2)\dfrac{d}{dx} \phi(x) \right]+\left(\dfrac{\dfrac{1}{4}-\alpha^2}{x^2}-c^2x^2\right)\phi(x).
 \end{equation}
 Hence,  $\varphi^{\alpha}_{n,c}$ is the $n-$th order bounded   eigenfunction of the operator $-\mathcal D^{\alpha}_c,$  associated with the eigenvalue $\chi_{n,\alpha}(c),$ that is
 \begin{equation}
 \label{differ_operator2}
 -\dfrac{d}{dx} \left[ (1-x^2)\dfrac{d}{dx} \varphi^{\alpha}_{n,c}(x) \right] - \left( \dfrac{\dfrac{1}{4}-\alpha^2}{x^2}-c^2x^2 \right)\varphi^{\alpha}_{n,c}(x)=\chi_{n,\alpha}(c)\varphi^{\alpha}_{n,c}(x),\quad x\in [0,1].
 \end{equation}
 The Hankel prolate functions form an orthonormal basis of $L^2(0,1).$ 
 Moreover, they form an orthogonal basis of the Hankel Paley Wiener space $\mathcal B_c^{\alpha}$, the space of functions from $L^2(0,\infty)$ with Hankel transforms supported on $[0,c]$, 
 \begin{equation}\label{eq0.5}
 \int_0^{\infty}\vp(y)\psi_{m,c}^{(\alpha)}(y)dy=\frac{1}{c(\mu_n^{\alpha}(c))^2}\delta_{n,m},~~~~~
 \mathcal H^{\alpha}(\vp)=\frac{1}{c\mu_{n,\alpha}(c)}\vp(\frac{.}{c})\chi_{[0,c]}.
 \end{equation}
 where $\mathcal H^{\alpha}$ is the Hankel transform defined on $L^2(0,\infty)$ with kernel $H^{\alpha}(x,y)=\sqrt{xy}J_{\alpha}(xy).$ For more details about circular prolate, we refer the readers to \cite{Slepian3}, \cite{Karoui-Boulsane}, \cite{B.M}, \cite{M.B} and \cite{B. J. S}. 
\section{Non-asymptotic behaviour of eigenvalue}
\begin{theorem}
	Let $0<\omega<1$, $\alpha>-1/2$ be two real numbers and $N\geq 1$ be an integer, then for every $\frac{e\omega s^{(\alpha)}_{N+1}}{4}\leq n\leq N-1 $, we have
	\begin{equation}
	\tilde{\lambda}_{n,\alpha}^N(c)\lesssim\sqrt{\frac{2}{e}}\frac{1}{\zeta}\frac{1}{\sqrt{(2N+\alpha+3/2)}}\left(\frac{e\omega s_{N+1}^{(\alpha)}}{4n+2\alpha+3}\right)^{2n+\alpha+1/2}.
	\end{equation}
	Here $\zeta=\frac{1}{2}\displaystyle\min_{1\leq i\leq N}\left\{\left|s_{i+1}^{(\alpha)}-s_i^{(\alpha)}\right|\right\}.$
\end{theorem}
\begin{proof}
	Let's first give an explicit formula for $\tilde{\lambda}_{n,\alpha}^N(\omega)$ using the min-max theorem concerning the different positive eigenvalues associated to a self adjoint compact operator. More precisely, from Courant–Fischer–Weyl Min-Max variational principle,we have $$\tilde{\lambda}_{n,\alpha}^N(\omega)=\min_{S_n}\max_{f\in S_n^{\perp},\norm{f}=1}\scal{\tilde{\mathcal{Q}}^{\alpha}_{N,\omega}(f), f}_{L^2(0,\omega)}$$
	where $S_n$ is a subspace of $L^2(0,\omega)$ $n$-dimensional.
	 On the other hand, it's knowing that the family  $$T_{n,\omega}^{\alpha}(x)=\sqrt{\frac{2(2n+\alpha+1)}{\omega}}\left(\frac{x}{\omega}\right)^{\alpha+1/2}P_n^{(\alpha,0)}(1-2\left(\frac{x}{\omega}\right)^2),\,x\in(0,\omega),$$  is an orthonormal basis of $L^2(0,\omega)$,
	here $P_n^{(\alpha,0)}$ is the Jacobi polynomial of the first kind with order $(\alpha,0).$ The most important property of modified Jacobi polynomial is its relation to the Hankel transform $\mathcal{H}^{\alpha}$ defined in $L^2(0,\infty)$ with kernel $H(x,y)=\sqrt{xy}J_{\alpha}(xy)$, that is for every $n\geq 0,$ we have
	\begin{equation}\label{e2}
	\mathcal{H}^{\alpha}(j_{n,\omega}^{\alpha})=\frac{1}{\sqrt{\omega}}T_{n,\omega}^\alpha\,\chi_{(0,\omega)},
	\end{equation}
	where $j_{n,\omega}^\alpha$ is the the Spherical Bessel function defined by $$j^{\alpha}_{n,\omega}(x)=\sqrt{2(2n+\alpha+1)}\frac{J_{2n+\alpha+1}(\omega x)}{\sqrt{\omega x}},\,x\in(0,+\infty),$$ here $J_\beta$ is the Bessel function of the first kind with order $\beta>-1/2.$ Consequently, we conclude that the Spherical Bessel family form a complete orthonormal system of the Paley-Wiener space $\mathcal{H}B_{\omega}^\alpha$ associated to the Hankel transform defined by 
	\begin{equation}
	\mathcal{H}B_{\omega}^\alpha=\left \{ f\in L^2(0,\infty),\,\, \mbox{Support} \mathcal{H}^{\alpha}f\subseteq [0,\omega] \right\}.
	\end{equation}
	 We consider $S_n$ the span of $\{T^\alpha_{0,\omega},...,T^\alpha_{n-1,\omega}\}$ and $f\in S_n^{\perp}$ that is $f=\sum_{k\geq n}a_k(f)T^\alpha_{k,\omega}$ and $\norm{f}^2_{L^2(0,\omega)}=\sum_{k=n}^{+\infty}|a_k(f)|^2=1.$ Let's first give an explicit form for $	\tilde{\mathcal{Q}}^{\alpha}_{N,\omega}(T^\alpha_{k,\omega})$. Indeed, by using \eqref{e2}, we get that for every $k\geq 0$ and $x\geq 0,$
	\begin{eqnarray*}
	\tilde{\mathcal{Q}}^{\alpha}_{N,\omega}(T^\alpha_{k,\omega})(x)&=&\int_0^\omega\sum_{j=1}^{N}\phi^{(\alpha)}_j(x)\phi^{(\alpha)}_j(y)T^{\alpha}_{k,\omega}(y)dy\\&=&\sum_{j=1}^{N}\phi^{(\alpha)}_j(x)\int_0^{\omega}T^{\alpha}_{k,\omega}(y)\phi^\alpha_{j}(y)dy\\&=&\sqrt{2\omega}\sum_{j=1}^{N}\frac{\phi^{(\alpha)}_j(x)}{\sqrt{s_j^{(\alpha)}}|J_{\alpha+1}(s_j^{(\alpha)})|}\int_0^{+\infty}\sqrt{s_j^{(\alpha)}y}J_{\alpha}(s_j^{(\alpha)}y)\frac{1}{\sqrt{\omega}}T^{\alpha}_{k,\omega}( y)\chi_{(0,\omega)}(y)dy\\&=&\sqrt{2\omega}\sum_{j=1}^{N}\frac{j^{\alpha}_{k,\omega}(s_j^{(\alpha)})}{\sqrt{s_j^{(\alpha)}}|J_{\alpha+1}(s_j^{(\alpha)})|}\phi^{(\alpha)}_j(x).
	\end{eqnarray*}
From the Cauchy Schwartz inequality and taking into account that $(\phi^{(\alpha)}_j)$ is an orthonormal basis of $L^2(0,1)$, we have
\begin{eqnarray*}
\norm{\tilde{\mathcal{Q}}^{\alpha}_{N,\omega}(T^{\alpha}_{k,\omega})}_{L^2(0,\omega)}&\leq&2\omega\sum_{j=1}^{N}\frac{\left|j^{\alpha}_{k,\omega}(s_j^{(\alpha)})\right|}{\sqrt{s_j^{(\alpha)}}|J_{\alpha+1}(s_j^{(\alpha)})|}.
\end{eqnarray*}
Let $\zeta=\frac{1}{2}\displaystyle\min_{1\leq i\leq N}\left\{\left|s_{i+1}^{(\alpha)}-s_i^{(\alpha)}\right|\right\}$. By using Batir inequality given in \cite{B.N} and the famous inequality given by $$|J_{\alpha}(x)|\leq\frac{|\frac{x}{2}|^\alpha}{\Gamma(\alpha+1)},\,x\in\mathbb{R}$$ with $\left|\sqrt{s_j^{(\alpha)}}J_{\alpha+1}(s_j^{(\alpha)})\right|^{-1}=O(1),$ see \cite{Watson}, we obtain
\begin{eqnarray*}
\norm{\tilde{\mathcal{Q}}^{\alpha}_{N,\omega}(T^{\alpha}_{k,\omega})}_{L^2(0,\omega)}&\lesssim&2\omega\sqrt{2(2k+\alpha+1)}\sum_{j=1}^{N}\left|\frac{J_{2k+\alpha+1}(\omega s_j^{(\alpha)})}{\sqrt{\omega s_j^{(\alpha)}}}\right|\\&\lesssim&2\omega\sqrt{2(2k+\alpha+1)}\sum_{j=1}^{N}\left|\frac{(\omega s_j^{(\alpha)})^{2k+\alpha+1/2}}{2^{2k+\alpha+1}\Gamma(2k+\alpha+2)}\right|\\&\lesssim&2\frac{\sqrt{(2k+\alpha+1)}}{\Gamma(2k+\alpha+2)}\left(\frac{\omega }{2}\right)^{(2k+\alpha+3/2)}\frac{1}{\zeta}\sum_{j=1}^{N}(s_{j+1}^{(\alpha)}-s_{j}^{(\alpha)})\left(s_j^{(\alpha)}\right)^{(2k+\alpha+1/2)}\\&\lesssim&2\frac{1}{\zeta}\frac{\sqrt{(2k+\alpha+1)}}{(2k+\alpha+3/2)\Gamma(2k+\alpha+2)}\left(\frac{\omega s_{N+1}^{(\alpha)} }{2}\right)^{(2k+\alpha+3/2)}\\&\lesssim&\frac{1}{\zeta}\frac{\sqrt{2}}{\sqrt{e(2k+\alpha+3/2)}}\left(\frac{e\omega s_{N+1}^{(\alpha)}}{2(2k+\alpha+3/2)}\right)^{(2k+\alpha+3/2)}
\end{eqnarray*}
Finally, from the series expansion of $f\in S_n^{\perp}$, that is $f=\sum_{k\geq n}a_k(f)T^{\alpha}_{k,\omega}$ and by using the Cauchy-Schwartz inequality, one gets
\begin{eqnarray*}
\left|\scal{\tilde{\mathcal{Q}}^{\alpha}_{N,\omega}(f), f}_{L^2(0,\omega)}\right|&\leq&\sum_{k\geq n}|a_k(f)| \left|\scal{\tilde{\mathcal{Q}}^{\alpha}_{N,c}(T^{\alpha}_{k,\omega}), f}_{L^2(0,\omega)}\right|\\&\leq&\sum_{k\geq n}\norm{\tilde{\mathcal{Q}}^{\alpha}_{N,c}(T^{\alpha}_{k,\omega})}_{L^2(0,\omega)}\\&\lesssim&\frac{1}{\zeta}\frac{\sqrt{2}}{\sqrt{e(2N+\alpha+3/2)}}\sum_{k\geq n}\left(\frac{e\omega s_{N+1}^{(\alpha)}}{2(2k+\alpha+3/2)}\right)^{2k+\alpha+3/2}\\&\lesssim&\frac{1}{\zeta}\frac{\sqrt{2}}{\sqrt{e(2N+\alpha+3/2)}}\left(\frac{e\omega s_{N+1}^{(\alpha)}}{4n+2\alpha+3}\right)^{2n+\alpha+3/2}.
\end{eqnarray*}
Consequently, from the Courant–Fischer–Weyl Min-Max variational principle one gets the desired result.
\end{proof}
\begin{remark}
According to the given details from the energy maximisation problem \eqref{e0} and the Min-Max theorem, we can easily check that for every $0\leq n\leq N-1$ and $S_n$ subspace of $S_N$ with dimensional $n,$ we have
\begin{equation}\label{e3} \tilde{\lambda}_{n,\alpha}^N(\omega)=\begin{cases}\displaystyle\max_{x\in S_N}\frac{\|\hat{x}\|^2_{L^2([0,\omega])}}{\|\hat{x}\|^2_{L^2(0,1)}} & \mbox{ if } n=0\\\displaystyle\max_{S_n}\min_{x\in S_n}\frac{\|\hat{x}\|^2_{L^2([0,\omega])}}{\|\hat{x}\|^2_{L^2(0,1)}} & \mbox{ if } n\geq 1\end{cases}. \end{equation}	
\end{remark}
The following proposition and its proof are inspired from the similar theorem given in \cite{B.B.K} in the case of the classical discrete eigenvalues. However, the following proposition give us more precision around the decay quality of the eigenvalues $(\tilde{\lambda}_{n,\alpha}^N(\omega))$ using a simple comparison with the decay rate of eigenvalues $(\lambda_{n}^{\alpha}(c))$ associated to the integral transform $\mathcal{Q}_c^{\alpha}$ given in \cite{M.B}.
\begin{proposition}
	Let $0<\omega<1$, $\alpha>0$ be two real numbers and $N\geq 1$ be an integer. We denote $\zeta=\displaystyle\frac{1}{2}\min_{1\leq i\leq N-1}\left\{\left|s_{i+1}^{(\alpha)}-s_i^{(\alpha)}\right|\right\}$ and   $\varepsilon=\min\{\zeta,\,s_1^{(\alpha)}\}$, then for every $0\leq n\leq N-1 $, we have
	\begin{equation}
	\left(C'_{\alpha,1}(\omega)\right)^2 \lambda_{n}^{\alpha}\left(c\right)\leq\tilde{\lambda}_{n,\alpha}^N(\omega)\leq \left(C'_{\alpha,2}(\omega)\right)^2 \lambda_{n}^{\alpha}\left(c\right),
	\end{equation}
	where $c=\omega (s_N^{(\alpha)}+\varepsilon),$  $C'_{\alpha,1}(\omega)=\frac{2}{\sqrt{\varepsilon}}\frac{\Gamma(\alpha+1)}{\Gamma(\alpha+1/2)}m_\alpha$ and $C'_{\alpha,2}(\omega)=\frac{1}{\sqrt{\varepsilon}}\frac{\left(\varepsilon\omega\right)^{\alpha}}{J_{\alpha(\varepsilon\alpha)}}M_\alpha.$ Here $J_\alpha$ is the Bessel function of the first kind with order $\alpha$ and \begin{equation*} \begin{cases}M_\alpha=2^{(2\alpha-1/2)}\chi_{\{\alpha\geq 1/2\}}(\alpha)+\frac{2^{2-5\alpha}}{\sqrt{\alpha}}\chi_{\{0<\alpha<1/2\}}(\alpha)\\m_\alpha=\frac{2^{(2\alpha-1/2)}}{\sqrt{2\alpha}}\chi_{\{\alpha\geq 1/2\}}(\alpha)+2^{\alpha-1/2}\chi_{\{0<\alpha<1/2\}}(\alpha)\end{cases}.\end{equation*}
\end{proposition}
\begin{proof}
Let $\varepsilon>0$, we denoted by $\mathcal{F}_N$ the span of the first $N$ modified Bessel functions $\{\phi^{(\alpha)}_{j}.\frac{J_{\alpha}(\varepsilon.)}{y^{\alpha}},\,1\leq j\leq N\}.$ From the famous formula given in \cite{Watson} by $$\int_0^{+\infty}tJ_{\alpha}(at)J_{\alpha}(bt)J_{\alpha}(ct)\frac{dt}{t^{\alpha}}=\frac{2^{\alpha-1}A^{\alpha-\frac{1}{2}}}{\pi^{1/2}(abc)^{\alpha}\Gamma(\alpha+1/2)}\chi_{(|b-c|,\,b+c)}(a),$$
where $A=s(s-a)(s-b)(s-c)$ and $s=\frac{a+b+c}{2},$ we have for $a=x, b=s_j^{(\alpha)}$ and $c=\varepsilon,$ $$\mathcal{H}^{\alpha}\left(\frac{J_{\alpha}(\varepsilon.)}{y^{\alpha}}\phi^{(\alpha)}_{j}\right)(x)=\frac{2^{\alpha-\frac{1}{2}}\sqrt{x}}{\pi^{1/2}\Gamma(\alpha+1/2)|J_{\alpha+1}(s_j^{(\alpha)})|}\frac{\left(A(x)\right)^{\alpha-\frac{1}{2}}}{(x.\varepsilon .s_j^{(\alpha)})^{\alpha}}\chi_{\left(|s_j^{(\alpha)}-\varepsilon|, \,s_j^{(\alpha)}+\varepsilon\right)}(x),$$ where $A(x)=\frac{\left((\varepsilon+s_j^{(\alpha)})^2-x^2\right)\left(x^2-(s_j^{(\alpha)}-\varepsilon)^2\right)}{2^4}.$ Consequently, the Hankel transform of every function $f\in\mathcal{F}_N$ has support in $[|s_1^{(\alpha)}-\varepsilon|,\,s_N^{(\alpha)}+\varepsilon]\subset[0,s_N^{(\alpha)}+\varepsilon].$ Hence, $\mathcal{F}_N$ is a subspace of the Hankel Paley-Wiener space $\mathcal{HB}_{s_N^{(\alpha)}+\varepsilon}^{\alpha}.$ On the other hand, from a straightforward computation, we can easily check that for every $\gamma>0$ and $X\in(A,B)$, we have 
$$\frac{(X^\gamma-A^\gamma)(B^\gamma-X^\gamma)}{(B^\gamma-A^\gamma)}\leq (\gamma\vee 1)X^\gamma\frac{\sqrt{B}-\sqrt{A}}{\sqrt{B}+\sqrt{A}},$$ 
hence, we get the following inequality
\begin{eqnarray}\label{e14} \frac{x\left((\varepsilon+s_j^{(\alpha)})-x\right)\left(x-(s_j^{(\alpha)}-\varepsilon)\right)}{2^3}\leq\left(\frac{A(x)}{x}\right)&\leq& \frac{1}{2}\frac{\sqrt{s_j^{(\alpha)}+\varepsilon}-\sqrt{|s_j^{(\alpha)}-\varepsilon|}}{\sqrt{s_j^{(\alpha)}+\varepsilon}+\sqrt{|s_j^{(\alpha)}-\varepsilon|}}(s_j^{(\alpha)}.\varepsilon)\,x\nonumber\\&\leq& \frac{1}{2}\frac{s_j^{(\alpha)}+\varepsilon-|s_j^{(\alpha)}-\varepsilon|}{\left(\sqrt{s_j^{(\alpha)}+\varepsilon}+\sqrt{|s_j^{(\alpha)}-\varepsilon|}\right)^2}(s_j^{(\alpha)}.\varepsilon)\,x\nonumber\\&\leq& \varepsilon\min{\{s_j^{(\alpha)},\,\varepsilon\}}\,x.
\end{eqnarray}
Consequently, for the case when $\alpha\geq 1/2$ and from the previous inequality, we deduce that
\begin{eqnarray*}
\norm{\mathcal{H}^{\alpha}\left(\frac{J_{\alpha}(\varepsilon.)}{y^{\varepsilon}}\phi^{(\alpha)}_{j}\right)}_{L^2(0,+\infty)}&=&\frac{2^{\alpha-\frac{1}{2}}}{\pi^{1/2}\Gamma(\alpha+1/2)|J_{\alpha+1}(s_j^{(\alpha)})|(\varepsilon .s_j^{(\alpha)})^{\alpha}}\left[\int_{|s_j^{(\alpha)}-\varepsilon|}^{s_j^{(\alpha)}+\varepsilon}\frac{\left(A(x)\right)^{2\alpha-1}}{x^{2\alpha-1}}dx\right]^{1/2}\\&\leq&\frac{\left(2\varepsilon\min{\{s_j^{(\alpha)},\,\varepsilon\}}\right)^{\alpha-1/2}}{\pi^{1/2}\Gamma(\alpha+1/2)|J_{\alpha+1}(s_j^{(\alpha)})|(\varepsilon .s_j^{(\alpha)})^{\alpha}}\left[\frac{\left(s_j^{(\alpha)}+\varepsilon\right)^{2\alpha}-\left(|s_j^{(\alpha)}-\varepsilon|\right)^{2\alpha}}{2\alpha}\right]^{1/2}\\&\leq&\frac{\left(2\varepsilon\min{\{s_j^{(\alpha)},\,\varepsilon\}}\right)^{\alpha-1/2}}{\pi^{1/2}\Gamma(\alpha+1/2)|J_{\alpha+1}(s_j^{(\alpha)})|(\varepsilon .s_j^{(\alpha)})^{\alpha}}\left(2\min{\{s_j^{(\alpha)},\,\varepsilon\}}(s_j^{(\alpha)}+\varepsilon)^{2\alpha-1}\right)^{1/2}.
\end{eqnarray*}
Let $0<\varepsilon<s_1^{(\alpha)}<..<s_j^{(\alpha)}..$ such that $\varepsilon<\zeta=\displaystyle\frac{1}{2}\min_{1\leq i\leq N-1}\left\{\left|s_{i+1}^{(\alpha)}-s_i^{(\alpha)}\right|\right\}$, then from the fact that $(s_j^{(\alpha)})^{-1/2}|J_{\alpha+1}(s_j^{(\alpha)})|^{-1}=\sqrt{\frac{\pi}{2}}(1+O(\frac{1}{j}))$, see \cite{Watson}, we obtain
\begin{eqnarray}\label{e-1}
\frac{2^{(2\alpha-1/2)}}{\sqrt{2\alpha}}\times\frac{\varepsilon^{\alpha-1/2}}{\Gamma(\alpha+1/2)}\leq\,\norm{\mathcal{H}^{\alpha}\left(\frac{J_{\alpha}(\varepsilon.)}{y^{\alpha}}\phi^{(\alpha)}_{j}\right)}_{L^2(0,+\infty)}&\leq& 2^{(2\alpha-1/2)}\times\frac{\varepsilon^{\alpha-1/2}}{\Gamma(\alpha+1/2)},
\end{eqnarray}
Let's focus on the case when $0<\alpha<1/2$. For $x\in\left(|\varepsilon-s_j^{(\alpha)}|,(\varepsilon+s_j^{(\alpha)})\right)$ and from \eqref{e14}, we have
\begin{eqnarray*}
\frac{1}{\varepsilon\min{\{s_j^{(\alpha)},\,\varepsilon\}}\,x}\leq\frac{x}{A(x)}&\leq& \frac{2^3}{|\varepsilon-s_j^{(\alpha)}|\max\{\varepsilon, s_j^{(\alpha)}\}}\left(\frac{x}{(\varepsilon+s_j^{(\alpha)}-x)(x-|\varepsilon-s_j^{(\alpha)}|)}\right)
\end{eqnarray*}
which leads to
\begin{eqnarray*}
\norm{\mathcal{H}^{\alpha}\left(\frac{J_{\alpha}(\varepsilon.)}{y^{\varepsilon}}\phi^{(\alpha)}_{j}\right)}_{L^2(0,+\infty)}&=&\frac{2^{\alpha-\frac{1}{2}}}{\pi^{1/2}\Gamma(\alpha+1/2)|J_{\alpha+1}(s_j^{(\alpha)})|(\varepsilon .s_j^{(\alpha)})^{\alpha}}\left[\int_{|s_j^{(\alpha)}-\varepsilon|}^{s_j^{(\alpha)}+\varepsilon}\left(\frac{x}{A(x)}\right)^{1-2\alpha}dx\right]^{1/2}\nonumber\\&=&\frac{2^{\alpha-\frac{1}{2}}}{\pi^{1/2}\Gamma(\alpha+1/2)|J_{\alpha+1}(s_j^{(\alpha)})|(\varepsilon .s_j^{(\alpha)})^{\alpha}}\left[\left(\int_{|s_j^{(\alpha)}-\varepsilon|}^{s_j^{(\alpha)}}+\int_{s_j^{(\alpha)}}^{s_j^{(\alpha)}+\varepsilon}\right)\left(\frac{x}{A(x)}\right)^{1-2\alpha}dx\right]\nonumber\\&\leq&\frac{2^{(1-2\alpha)}}{\Gamma(\alpha+1/2)\varepsilon^{\alpha}(s_j^{(\alpha)}-\varepsilon)^{1/2-\alpha}}\times \frac{(s_j^{(\alpha)}+\varepsilon)^{1/2-\alpha}}{\sqrt{\alpha}}\times \left(\frac{\varepsilon}{2}\right)^{2\alpha-1/2}\nonumber\\&\leq&\frac{2^{2-5\alpha}}{\sqrt{\alpha}}\times \frac{\varepsilon^{\alpha-1/2}}{\Gamma(\alpha+1/2)}.
\end{eqnarray*}
and 
\begin{eqnarray*}
\norm{\mathcal{H}^{\alpha}\left(\frac{J_{\alpha}(\varepsilon.)}{y^{\varepsilon}}\phi^{(\alpha)}_{j}\right)}_{L^2(0,+\infty)}&=&\frac{2^{\alpha-\frac{1}{2}}}{\pi^{1/2}\Gamma(\alpha+1/2)|J_{\alpha+1}(s_j^{(\alpha)})|(\varepsilon .s_j^{(\alpha)})^{\alpha}}\left[\int_{|s_j^{(\alpha)}-\varepsilon|}^{s_j^{(\alpha)}+\varepsilon}\left(\frac{x}{A(x)}\right)^{1-2\alpha}dx\right]^{1/2}\nonumber\nonumber\\&\geq&2^{\alpha-1/2}\times\frac{\varepsilon^{\alpha-1/2}}{\Gamma(\alpha+1/2)}.
\end{eqnarray*}
Hence, we get
\begin{equation}\label{e-2}
2^{\alpha-1/2}\times\frac{\varepsilon^{\alpha-1/2}}{\Gamma(\alpha+1/2)}\leq \norm{\mathcal{H}^{\alpha}\left(\frac{J_{\alpha}(\varepsilon.)}{y^{\varepsilon}}\phi^{(\alpha)}_{j}\right)}_{L^2(0,+\infty)}\leq \frac{2^{2-5\alpha}}{\sqrt{\alpha}}\times \frac{\varepsilon^{\alpha-1/2}}{\Gamma(\alpha+1/2)}.
\end{equation}
Finally, from \eqref{e-1}, \eqref{e-2} and for every $\alpha>0,$ we easily check that 
\begin{equation}\label{e-3}
m_\alpha\frac{\varepsilon^{\alpha-1/2}}{\Gamma(\alpha+1/2)}\leq \norm{\mathcal{H}^{\alpha}\left(\frac{J_{\alpha}(\varepsilon.)}{y^{\varepsilon}}\phi^{(\alpha)}_{j}\right)}_{L^2(0,+\infty)}\leq M_\alpha\frac{\varepsilon^{\alpha-1/2}}{\Gamma(\alpha+1/2)},
\end{equation}
where
\begin{equation*} \begin{cases}M_\alpha=2^{(2\alpha-1/2)}\chi_{\{\alpha\geq 1/2\}}(\alpha)+\frac{2^{2-5\alpha}}{\sqrt{\alpha}}\chi_{\{0<\alpha<1/2\}}(\alpha)\\m_\alpha=\frac{2^{(2\alpha-1/2)}}{\sqrt{2\alpha}}\chi_{\{\alpha\geq 1/2\}}(\alpha)+2^{\alpha-1/2}\chi_{\{0<\alpha<1/2\}}(\alpha)\end{cases}.\end{equation*}
 Let $f\in\mathcal{F}_N$ then $f(y)=\left(\sum_{j=1}^{N}a_j\phi^{(\alpha)}_{j}(y)\right)\frac{J_{\alpha}(\varepsilon.)}{y^{\alpha}},\,y\in(0,+\infty).$ From Parseval's equality, \eqref{e-3} and the fact that $\mbox{Support}\left(\mathcal{H}^{\alpha}\left(\frac{J_{\alpha}(\varepsilon.)}{y^{\alpha}}\phi^{(\alpha)}_{i}\right)\right)\cap \mbox{Support}\left(\mathcal{H}^{\alpha}\left(\frac{J_{\alpha}(\varepsilon.)}{y^{\alpha}}\phi^{(\alpha)}_{j}\right)\right)=\varnothing,\,\,i\not=j,\,(\varepsilon<\zeta)$, we obtain,
 \begin{eqnarray}\label{e-4}
  C^2_{\alpha,1}(\varepsilon)\norm{\hat{x}}^2_{L^2(0,1)}\leq\norm{f}^2_{L^2(0,\infty)}&=&\sum_{j=1}^N|a_j|^2\norm{\mathcal{H}^{\alpha}\left(\frac{J_{\alpha}(\varepsilon.)}{y^{\alpha}}\phi^{(\alpha)}_{j}\right)}^2_{L^2(0,+\infty)}\leq C^2_{\alpha,2}(\varepsilon)\norm{\hat{x}}^2_{L^2(0,1)}
  \end{eqnarray}
  where $C_{\alpha,1}(\varepsilon)=m_\alpha\times\frac{\varepsilon^{\alpha-1/2}}{\Gamma(\alpha+1/2)},\,C_{\alpha,2}(\varepsilon)=M_\alpha\times\frac{\varepsilon^{\alpha-1/2}}{\Gamma(\alpha+1/2)}$ and $x=\left(a_n\chi_{[|1,N|]}(n)\right)_n\in S_N$. 
  On the other hand, taking into account the decay of the function $x\to\left(\frac{J_{\alpha}( x)}{x^{\alpha}}\right)^2$ on $(0,s_1^{(\alpha)})$ and the fact that $\varepsilon< s_1^{(\alpha)},$ we get
  \begin{eqnarray}\label{e-5}
  \frac{J^2_{\alpha}(\varepsilon\omega )}{\omega^{2\alpha}}\norm{\hat{x}}^2_{L^2(0,\omega)}\leq\norm{f}^2_{L^2(0,\omega)}&=&\int_0^{\omega}|\hat{x}(y)|^2\left|\frac{J_{\alpha}(\varepsilon y)}{y^{\alpha}}\right|^2dy\leq \frac{\varepsilon^{2\alpha}}{2^{2\alpha}\Gamma^2(\alpha+1)}\norm{\hat{x}}^2_{L^2(0,\omega)}
  \end{eqnarray}
  For $g(x)=f(\omega x),\,x\in(0,1)$ and from \eqref{e-4} with \eqref{e-5}, we have
\begin{equation*}
\left(C'_{\alpha,1}(\omega)\right)^2\frac{\norm{g}^2_{L^2(0,1)}}{\norm{g}^2_{L^2(0,\infty)}}\leq  \frac{\norm{\hat{x}}^2_{L^2(0,\omega)}}{\norm{\hat{x}}^2_{L^2(0,1)}}\leq\left(C'_{\alpha,2}(\omega)\right)^2\frac{\norm{g}^2_{L^2(0,1)}}{\norm{g}^2_{L^2(0,\infty)}},
\end{equation*}
where $C'_{\alpha,1}(\omega)=\left(\frac{2}{\varepsilon}\right)^{\alpha}\Gamma(\alpha+1)C_{\alpha,1}(\omega)$ and $C'_{\alpha,2}(\omega)=\frac{\omega^{\alpha}}{J_\alpha(\varepsilon\alpha)}C_{\alpha,2}(\omega).$ On the other hand, $\mathcal{H}^{\alpha}(g)=\frac{1}{\omega}\mathcal{H}^{\alpha}(f)(\frac{.}{\omega})$, then we guarantee that $\mbox{Support}(\mathcal{H}^{\alpha}(g))\subset [0,\omega (s_N^{(\alpha)}+\varepsilon)]$ which we easily conclude that $g\in\mathcal{HB}_{\omega (s_N^{(\alpha)}+\varepsilon)}^{\alpha}.$ Finally by using \eqref{e3} (Min-Max characterisation), we get the desired result.
\end{proof}
In the following, we will focus on the comparison between the spectrum of  $\tilde{\mathcal{Q}}^{\alpha}_{N,\omega}$ and $\mathcal{Q}_c^{\alpha}$. Moreover, we will devote our efforts to estimating the distribution of significant eigenvalues on the interval $(0,1).$  The claims that we need are provided by the following proposition inspired from \cite{Ro} and \cite{Watson} chapter XVII.

\begin{proposition}
Let $\alpha\geq-1/2$ and $0<\omega<1$ then for every $N\geq 1$ and $(x,y)\in(0,\omega)^2,$ we have
\begin{equation}\label{e6}
\tilde{K}^{(\alpha)}_{N,\omega}(x,y)=K_{c_N}^{\alpha}(x,y)+F_{c_N}^{\alpha}(x,y)+O\left(\frac{1}{c_N}\right),
\end{equation}
where $K_{c_N}^{\alpha}$ is the kernel associated with the operator $\mathcal{Q}_c^{\alpha}$ given in section $3$ and $F_{c_N}^{\alpha}$ is given by the following identity, 
\begin{equation}
F_{c_N}^{\alpha}(x,y)=\frac{2}{\pi}\sin((x+y)c_N-2\gamma_\alpha)V(x+y)+\frac{2}{\pi}\sin((x-y)c_N)V(x-y).
\end{equation}
with $$\frac{1}{2}\displaystyle\frac{r}{4-r^2}\leq V(r)=\int_0^{\infty}\frac{\sinh(rt)}{1+e^{-2t}}e^{-2t}dt\leq \displaystyle\frac{r}{4-r^2}.$$	
\end{proposition}
\begin{proof}
	Let $H^{(1)}_\alpha$ be the Bessel function of the third kind called Hankel function, we consider the holomorphic function given by $$F(z)=\frac{zH^{(1)}_\alpha(z)J_{\alpha}(xz)J_{\alpha}(yz)}{J_{\alpha}(z)},\,\,z=\xi+i\eta,$$ within the contour consisting of the $\xi-$axis from $-c_N$ to $c_N,$ indented at $\xi=0$ and at the zero of $J_{\alpha}$ with radius $\varepsilon>0$ and the lines $\xi=c_N,\,\eta=M>0$ and $\xi=-c_N.$ From the Residue theorem, we have
	$$\int F(z)dz=0.$$
	On the other hand, if we denoted by $s_0^{(\alpha)}=0,$ then we have
	\begin{eqnarray}\label{e7}
	\int F(z)dz&=&-i\varepsilon\sum_{j=1}^N\int_0^\pi F(\pm s_j^{(\alpha)}+\varepsilon e^{it})e^{it}dt-i\varepsilon\int_0^\pi F(\varepsilon e^{it})e^{it}dt\\&+&\int_{-c_N}^{-s_N^{(\alpha)}-\varepsilon}F(t)dt+\sum_{j=0}^{N-1}\int_{-s_{j+1}^{(\alpha)}+\varepsilon}^{-s_{j}^{(\alpha)}-\varepsilon} F(t)dt+\sum_{j=0}^{N-1}\int_{s_j^{(\alpha)}+\varepsilon}^{s_{j+1}^{(\alpha)}-\varepsilon} F(t)dt+\int_{s_N^{(\alpha)}-\varepsilon}^{c_N}F(t)dt\nonumber\\&+&i\int_0^M F(c_N+it)dt-\int_{-c_N}^{c_N}F(t+iM)dt-i\int_0^M F(-c_N+it)dt\nonumber.
	\end{eqnarray}
At first, for every $0\leq j\leq N-1,$ we have
$$\varepsilon F(\pm s_j^{(\alpha)}+\varepsilon e^{it})e^{it}=\frac{(\pm s_j^{(\alpha)}+\varepsilon e^{it})H^{(1)}_\alpha(\pm s_j^{(\alpha)}+\varepsilon e^{it})J_{\alpha}(x(\pm s_j^{(\alpha)}+\varepsilon e^{it}))J_{\alpha}(y(\pm s_j^{(\alpha)}+\varepsilon e^{it}))}{\frac{J_{\alpha}(\pm s_j^{(\alpha)}+\varepsilon e^{it})-J_{\alpha}(\pm s_j^{(\alpha)})}{(\pm s_j^{(\alpha)}+\varepsilon e^{it})-(\pm s_j^{(\alpha)})}}$$
By tending $\varepsilon$ to zero, we obtain,
$$\displaystyle\lim_{\varepsilon\to0}\varepsilon F(\pm s_j^{(\alpha)}+\varepsilon e^{it})e^{it}= \frac{\pm s_j^{(\alpha)}H^{(1)}_\alpha(\pm s_j^{(\alpha)})J_{\alpha}(\pm s_j^{(\alpha)}x)J_{\alpha}(\pm s_j^{(\alpha)}y)}{J'_{\alpha}(\pm s_j^{(\alpha)})}$$
From the Wronskian formula given in \cite{Watson} by
$\mathcal{W}\{J_{\alpha}, H^{(1)}_\alpha\}(z)=\displaystyle\frac{2i}{\pi z}$ and the fact that $J_{\alpha}(e^{i\pi}z)=e^{i\alpha\pi}J_{\alpha}(z)$, we obtain
\begin{equation}
\displaystyle\lim_{\varepsilon\to 0}-i\varepsilon \int_0^{\pi}F(\pm s_j^{(\alpha)}+\varepsilon e^{it})e^{it}dt=-2\frac{J_{\alpha}(\pm s_j^{(\alpha)}x)J_{\alpha}(\pm s_j^{(\alpha)}y)}{\left(J'_{\alpha}(\pm s_j^{(\alpha)})\right)^2}=-2\frac{J_{\alpha}( s_j^{(\alpha)}x)J_{\alpha}( s_j^{(\alpha)}y)}{\left(J'_{\alpha}( s_j^{(\alpha)})\right)^2}
\end{equation}
Finally, from the previous equation and the fact that $\displaystyle\lim_{M\to+\infty}F(t+iM)=0$, then by tending $\varepsilon\to 0$ and $M\to +\infty$ in \eqref{e7}, we get
\begin{eqnarray*}
2\sum_{j=1}^N\frac{J_{\alpha}( s_j^{(\alpha)}x)J_{\alpha}( s_j^{(\alpha)}y)}{\left(J'_{\alpha}( s_j^{(\alpha)})\right)^2}&=&\frac{1}{2}\int_0^{c_N}(F(t)+F(-t))dt+\frac{i}{2}\int_0^{+\infty}(F(-c_N+it)-F(c_N+it))dt\\&=&\frac{1}{2}\int_0^{c_N}\frac{t\left(H^{(1)}_\alpha(t)+H^{(2)}_\alpha(t)\right)J_{\alpha}(xt)J_{\alpha}(yt)}{J_{\alpha}(t)}dt+\frac{i}{2}\int_0^{+\infty}(F(-c_N+it)-F(c_N+it))dt\\&=&\int_0^{c_N}tJ_{\alpha}(xt)J_{\alpha}(yt)dt+\frac{i}{2}\int_0^{+\infty}(F(-c_N+it)-F(c_N+it))dt\\&=&c_N\frac{xJ_{\alpha+1}(c_Nx)J_{\alpha}(c_Ny)-yJ_{\alpha+1}(c_Ny)J_{\alpha}(c_Nx)}{x^2-y^2}+\frac{i}{2}\int_0^{+\infty}\left(F(-c_N+it)-F(c_N+it)\right)dt.
\end{eqnarray*}
Using that $\displaystyle F(e^{i\pi}z)=\frac{zH^{(2)}_\alpha(z)J_{\alpha}(xz)J_{\alpha}(yz)}{J_{\alpha}(z)},$ where $H^{(2)}_\alpha$ is the Hankel function of the second kind, and the fact that $H^{(2)}_\alpha(\bar{z})=\overline{H^{(1)}_\alpha}(z),$ we get 
$F(-c_N+it)=\overline{F(c_N+it)}.$ Then we have
\begin{equation}
\tilde{K}^{(\alpha)}_{N,\omega}(x,y)=K_{c_N}^{\alpha}(x,y)+\sqrt{xy}\int_0^{+\infty}\Im\left(F(c_N+it)\right)dt.
\end{equation}
Using the following uniform estimate of $H^{(1)}_{\alpha}$ and $H^{(2)}_{\alpha}$ given in \cite{Watson} by
$$H^{(1)}_{\alpha}(z)=\sqrt{\frac{2}{\pi z}}e^{i\left(z-\gamma_{\alpha}\right)}\left(1+\frac{I_1(z)}{z}\right),\,\,H^{(2)}_{\alpha}(z)=\sqrt{\frac{2}{\pi z}}e^{-i\left(z-\gamma_{\alpha}\right)}\left(1+\frac{I_2(z)}{z}\right)$$
where $I_1$ and $I_2$ two bounded functions and $\gamma_{\alpha}=\frac{\alpha\pi}{2}+\frac{\pi}{4}$, and by taken account that for every $z$ we have $J_\alpha(z)=\displaystyle\frac{H^{(1)}_{\alpha}(z)+,H^{(2)}_{\alpha}(z)}{2}$, then we get
$$J_\alpha(z)=\sqrt{\frac{2}{\pi z}}\left[\cos(z-\gamma_\alpha)+\frac{e^{i\left(z-\gamma_{\alpha}\right)}I_1(z)+e^{-i\left(z-\gamma_{\alpha}\right)}I_2(z)}{2z}\right].$$
Finally, from the fact that $z=c_N+it,\,t\in(0,\infty)$, we obtain
\begin{eqnarray*}
\Im\left(F(c_N+it)\right)&=&\frac{2}{\pi\sqrt{xy}}\sin(2(c_N-\gamma_\alpha))\frac{\cos((x+y)c_N-2\gamma_\alpha)\cosh((x+y)t)+\cos((x-y)c_N)\cosh((x-y)t)}{(e^{-2t}+\cos(2(c_N-\gamma_\alpha)))^2+\sin^2(2(c_N-\gamma_\alpha))}e^{-2t}\\&+&\frac{2}{\pi\sqrt{xy}}\cos(2(c_N-\gamma_\alpha))\frac{\sin((x+y)c_N-2\gamma_\alpha)\sinh((x+y)t)+\sin((x-y)c_N)\sinh((x-y)t)}{(e^{-2t}+\cos(2(c_N-\gamma_\alpha)))^2+\sin^2(2(c_N-\gamma_\alpha))}e^{-2t}\\&+&\frac{2}{\pi\sqrt{xy}}\frac{\sin((x+y)c_N-2\gamma_\alpha)\sinh((x+y)t)+\sin((x-y)c_N)\sinh((x-y)t)}{(e^{-2t}+\cos(2(c_N-\gamma_\alpha)))^2+\sin^2(2(c_N-\gamma_\alpha))}e^{-4t}\\&+&\frac{\left(\cosh((x+y)t)+\cosh((x-y)t)\right)}{\sqrt{xy}}\frac{e^{-2t}}{1+e^{-2t}}O\left(\frac{1}{c_N}\right)
\end{eqnarray*}
We recall that $c_N=(N+\frac{\alpha}{2}+\frac{1}{4})\pi=N\pi+\gamma_{\alpha},$ then we have $c_N-\gamma_{\alpha}=N\pi.$ Consequently, we have
\begin{eqnarray*}
\sqrt{xy}\int_0^{+\infty}\Im\left(F(c_N+it)\right)dt&=&\frac{2}{\pi}\sin((x+y)c_N-2\gamma_\alpha)\int_0^{\infty}\frac{\sinh((x+y)t)}{1+e^{-2t}}e^{-2t}dt\\&+&\frac{2}{\pi}\sin((x-y)c_N)\int_0^{\infty}\frac{\sinh((x-y)t)}{1+e^{-2t}}e^{-2t}dt+O\left(\frac{1}{c_N}\right).
\end{eqnarray*}
We can easily prove that $\displaystyle\int_0^{\infty}\frac{\sinh(rt)}{1+e^{-2t}}e^{-2t}dt$ is controllable on both sides by 
$\displaystyle\frac{r}{4-r^2},\,\,|r|<2,$ which we can deduce that for every $(x,y)\in(0,\omega)^2,$ we have
\begin{eqnarray*}
\sqrt{xy}\int_0^{+\infty}\Im\left(F(c_N+it)\right)dt&=&\frac{2}{\pi}\sin((x+y)c_N-2\gamma_\alpha)V(x+y)\\&+&\frac{2}{\pi}\sin((x-y)c_N)V(x-y)+O\left(\frac{1}{c_N}\right).
\end{eqnarray*}
where $$\frac{1}{2}\displaystyle\frac{r}{4-r^2}\leq V(r)=\int_0^{\infty}\frac{\sinh(rt)}{1+e^{-2t}}e^{-2t}dt\leq \displaystyle\frac{r}{4-r^2}.$$
\end{proof}
\begin{proposition}
	Let $0<\omega<1$, $\alpha\geq-1/2$ be two real numbers and $N\geq 1$ be an integer, then we have
	\begin{equation}
	\norm{\tilde{\lambda}\left(\tilde{\mathcal{Q}}^{\alpha}_{N,\omega}\right)-\lambda\left(\mathcal{Q}_c^\alpha\right)}^2_{\ell_2(\mathbb{R})}:=\sum_{n=0}^{+\infty}\left|\tilde{\lambda}_{n,\alpha}^N(\omega)- \lambda_{n,\alpha}(c)\right|^2\leq \frac{2}{\pi}\sqrt{\ln\left(\frac{1}{(1-\omega^2)}\right)}+C\frac{\omega^2}{c}. 
	\end{equation} 
	where $c=(N+\frac{1}{2}\alpha+\frac{1}{4})\pi\omega.$ 
\end{proposition}
\begin{proof}
	We recall that the most time concentrated functions on $(0,\omega)$ and Hankel band limited with bandwidth $c_N=(N+\frac{1}{2}\alpha+\frac{1}{4})\pi$ was the eigenfunctions of the positive self adjoint operator $\mathcal{Q}_{c_N}^{\alpha}$ satisfying the following integral equation
	\begin{equation*}
	\mathcal{Q}_{c_N}^{\alpha}(f)=\int_0^{\omega}K_{c_N}^{\alpha}(.,y)f(y)dy=\lambda_{\omega}(c_N)f
	\end{equation*}
	By using a change of variable $y=\omega u$, we obtain that for all $x\in(0,1),$
	\begin{equation*}
	\int_0^{1}\omega K_{c_N}^{\alpha}(\omega x,\omega y)f(\omega y)dy=\lambda_{\omega}(c_N)f(\omega x).
	\end{equation*}
	Further more, $\mathcal{H}^{\alpha}(f(\omega.))=\frac{1}{\omega}\mathcal{H}^{\alpha}(f)(\frac{.}{\omega}),$ so if $f$ is Hankel supported in $(0,c_N)$ then Support$\left(\mathcal{H}^{\alpha}(f(\omega.))\right)\subset(0,\omega c_N).$ Hence, $\lambda_{\omega}(c_N)=\lambda(\omega c_N)$, where $\lambda$ are the eigenvalues of $\mathcal{Q}_{c}^{\alpha}$ given in section $3$ with $c=\omega c_N.$ 
	Let's begin from the Wielandt-Hoffman inequality which is valid for any compact operator, that is 
	\begin{equation}
	\norm{\tilde{\lambda}\left(\tilde{\mathcal{Q}}^{\alpha}_{N,\omega}\right)-\lambda\left(\mathcal{Q}_{c}^\alpha\right)}_{\ell_2(\mathbb{R})}=\norm{\tilde{\lambda}\left(\tilde{\mathcal{Q}}^{\alpha}_{N,\omega}\right)-\lambda\left(\mathcal{Q}_{c_N}^\alpha\right)}_{\ell_2(\mathbb{R})}\leq \norm{\tilde{\mathcal{Q}}^{\alpha}_{N,\omega}-\mathcal{Q}_{c_N}^\alpha}_{HS}.
	\end{equation}
	where $\norm{.}_{HS}$ is the Hilbert-Schmidt norm. Consequently, we must have to prove that $\norm{\tilde{\mathcal{Q}}^{\alpha}_{N,c}-\mathcal{Q}_{c_N}^\alpha}_{HS}$ satisfy our result. In order to prove that, we will use the explicit kernel of $\tilde{\mathcal{Q}}^{\alpha}_{N,c}-\mathcal{Q}_{c_N}^\alpha$ given by \eqref{e6},
then we obtain
	\begin{eqnarray*}
	\norm{\tilde{\mathcal{Q}}^{\alpha}_{N,\omega}-\mathcal{Q}_{c_N}^\alpha}_{HS}^2&=&\int_0^{\omega}\int_0^{\omega}\left|\tilde{K}^{\alpha}_{N,\omega}(x,y)-K_{c_N}^\alpha(x,y)\right|^2dxdy\\&\leq&\frac{4}{\pi^2}\int_0^{\omega}\int_0^{\omega}\left(\frac{(x+y)}{4-(x+y)^2}+\frac{(x-y)}{4-(x-y)^2}\right)^2dxdy+C\frac{\omega^2}{c^2_N}\\&\leq&\frac{1}{\pi^2}\int_0^{\omega}\int_0^{\omega}\left(\frac{1}{2-(x+y)}-\frac{1}{2+(x+y)}+\frac{1}{2-(x-y)}-\frac{1}{2+(x-y)}\right)^2dxdy+C\frac{\omega^2}{c^2_N}\\&\leq&\frac{4}{\pi^2}\ln(\frac{1}{1-w^2})+C\frac{\omega^2}{c^2_N}.
	\end{eqnarray*}
\end{proof}
\begin{lemma}
	Let $0<\omega<1$, $\alpha\geq-1/2$ be two real numbers and $N\geq N_0$ be an integer, then we have
	\begin{equation}
	\mbox{Trace}\left(\tilde{\mathcal{Q}}^{\alpha}_{N,\omega}\right)=\frac{c}{\pi}-\frac{\alpha}{2}+\left(\frac{1+\omega^2}{c}+\frac{1}{2\pi}\ln\left(\frac{1+\omega}{1-\omega}\right)\right)O(1).
	\end{equation}
	where $c=(N+\frac{1}{2}\alpha+\frac{1}{4})\pi\omega.$	
\end{lemma}
\begin{proof}
By the Mercer theorem, the trace of an integral operator $\mathcal{H}$ defined on $L^2(I,\omega)$ with associated kernel $H(x,y),\, x,y\in I,$ is of the form
\begin{equation*}
\mbox{Trace}\left(\mathcal{H}\right)=\int_IH(x,x)\omega(x)dx.
\end{equation*}
Then, we have
\begin{equation}
\mbox{Trace}\left(\tilde{\mathcal{Q}}^{\alpha}_{N,\omega}\right)=\int_0^{\omega}\tilde{K}^{\alpha}_{N,\omega}(x,x)dx
\end{equation}
From the estimate kernel given in \eqref{e6} and the previous equality, we obtain
\begin{eqnarray*}
\mbox{Trace}\left(\tilde{\mathcal{Q}}^{\alpha}_{N,\omega}\right)&=&\int_0^{\omega}K_{c_N}^{\alpha}(x,x)dx+\left(\frac{1}{2\pi}\int_0^{\omega}\left(\frac{1}{1-x}+\frac{1}{(1+x)}\right)dx+\frac{\omega^2}{c}\right) O(1)\\&=&\omega\int_0^{1}K_{c_N}^{\alpha}(\omega x,\omega x)dx+\left(\frac{1}{2\pi}\ln\left(\frac{1+\omega}{1-\omega}\right)+\frac{\omega^2}{c}\right)O(1)\\&=&\mbox{Trace}\left(\mathcal{Q}_{c}^\alpha\right)+\left(\frac{1}{2\pi}\ln\left(\frac{1+\omega}{1-\omega}\right)+\frac{\omega^2}{c}\right)O(1).
\end{eqnarray*}
where $c=(N+\frac{1}{2}\alpha+\frac{1}{4})\pi\omega.$ 
From \cite{M.B}, we have $\mbox{Trace}\left(\mathcal{Q}_{c}^\alpha\right)=\displaystyle\frac{c}{\pi}-\frac{\alpha}{2}+O(\frac{1}{c}),$ then we obtain the desired result.
\end{proof}
\begin{proposition}
Let $0<\omega<1$, $\alpha\geq-1/2$ be two real numbers and $N\geq 1$ be an integer, then for every $0<\varepsilon<1/2,$ we have
\begin{eqnarray*}
\#\{n,\,\varepsilon<\tilde{\lambda}_{n,\alpha}^N(\omega)<1-\varepsilon\}&\leq&\frac{K_{\alpha}}{\varepsilon(1-\varepsilon)}\left[\frac{\omega^2}{c}+\ln(c)+G(w)\right]\end{eqnarray*}
where
\begin{equation*}
 G(w)=(\frac{1}{2}+\omega)\ln\left(\frac{1+\omega}{1-\omega}\right)+\ln\left(\frac{4(1-w^2)}{4-\omega^2}\sqrt{\frac{2-\omega}{2+\omega}}\right)+2\frac{\omega(2+\omega)}{1-\omega^2}\ln\left(1+\frac{\omega}{2}\right),
\end{equation*}
 $c=\omega c_N=(N+\frac{1}{2}\alpha+\frac{1}{4})\pi\omega$ and $K_\alpha$ is a constant depending only of $\alpha.$	
\end{proposition}
\begin{proof}
Let's start by the following inequality which lead us to the wrights turn of the proof.
\begin{eqnarray*}
\varepsilon(1-\varepsilon)\sum_{\{n,\,\varepsilon<\tilde{\lambda}_{n,\alpha}^N(\omega)<1-\varepsilon\}}1&\leq&\sum_{n=0}^{N-1}\tilde{\lambda}_{n,\alpha}^N(\omega)(1-\tilde{\lambda}_{n,\alpha}^N(\omega))\\&=&\mbox{Trace}\left(\tilde{\mathcal{Q}}^{\alpha}_{N,\omega}\right)-\norm{\tilde{\mathcal{Q}}^{\alpha}_{N,\omega}}^2_{HS}
\end{eqnarray*}
where $\norm{\tilde{\mathcal{Q}}^{\alpha}_{N,\omega}}_{HS}$ is the Hilbert-Schmidt norm of $\tilde{\mathcal{Q}}^{\alpha}_{N,\omega}$. Hence, we get the desired first key of the proof, that is
\begin{eqnarray*}
\#\{n,\,\varepsilon<\tilde{\lambda}_{n,\alpha}^N(\omega)<1-\varepsilon\}&\leq&\frac{\mbox{Trace}\left(\tilde{\mathcal{Q}}^{\alpha}_{N,\omega}\right)-\norm{\tilde{\mathcal{Q}}^{\alpha}_{N,\omega}}^2_{HS}}{\varepsilon(1-\varepsilon)}\\&\leq&\frac{\left(\mbox{Trace}\left(\tilde{\mathcal{Q}}^{\alpha}_{N,\omega}\right)-\mbox{Trace}\left(\mathcal{Q}_{c}^\alpha\right)\right)+\left(\mbox{Trace}\left(\mathcal{Q}_{c}^\alpha\right)-\norm{\mathcal{Q}^{\alpha}_{c}}^2_{HS}\right)+\left(\norm{\mathcal{Q}^{\alpha}_{c}}^2_{HS}-\norm{\tilde{\mathcal{Q}}^{\alpha}_{N,\omega}}^2_{HS}\right)}{\varepsilon(1-\varepsilon)}.
\end{eqnarray*}
From the previous lemma, we have
 \begin{equation}\label{e8}
\mbox{Trace}\left(\tilde{\mathcal{Q}}^{\alpha}_{N,\omega}\right)-\mbox{Trace}\left(\mathcal{Q}_{c}^\alpha\right)=\left(\frac{1}{2\pi}\ln\left(\frac{1+\omega}{1-\omega}\right)+\frac{\omega^2}{c}\right)O(1),\end{equation} and from Abreu and Bandeira in \cite{Abreu} see also \cite{M.B}, we have \begin{equation}\label{e9}
\mbox{Trace}\left(\mathcal{Q}_{c}^\alpha\right)-\norm{\mathcal{Q}^{\alpha}_{c}}^2_{HS}\leq K\ln(c)+L.\end{equation}
So, it remaining to give an upper estimate of $\norm{\mathcal{Q}^{\alpha}_{c}}^2_{HS}-\norm{\tilde{\mathcal{Q}}^{\alpha}_{N,\omega}}^2_{HS}.$ On the other hand, from \eqref{e6} and for $c=\omega c_N$, we have
\begin{eqnarray*}
\norm{\tilde{\mathcal{Q}}^{\alpha}_{N,\omega}}^2_{HS}&=&\int_0^{\omega}\int_0^{\omega}\left(\tilde{K}^{\alpha}_{N,\omega}(x,y)\right)^2dxdy\\&=&\int_0^{\omega}\int_0^{\omega}\left(K^{\alpha}_{c_N}(x,y)\right)^2dxdy+\int_0^{\omega}\int_0^{\omega}\left(F^{\alpha}_{c_N}(x,y)\right)^2dxdy+2\int_0^{\omega}\int_0^{\omega}F^{\alpha}_{c_N}(x,y)K^{\alpha}_{N,\omega}(x,y)dxdy+O\left(\frac{\omega^2}{c_N}\right)\\&=&\norm{\mathcal{Q}^{\alpha}_{c}}^2_{HS}+\int_0^{\omega}\int_0^{\omega}\left(F^{\alpha}_{c_N}(x,y)\right)^2dxdy+2\int_0^{\omega}\int_0^{\omega}K^{\alpha}_{c_N}(x,y)F^{\alpha}_{c_N}(x,y)dxdy+O\left(\frac{\omega^2}{c_N}\right).
\end{eqnarray*}
It follows that 
\begin{eqnarray*}
\norm{\mathcal{Q}^{\alpha}_{c}}^2_{HS}-\norm{\tilde{\mathcal{Q}}^{\alpha}_{N,\omega}}^2_{HS}&=&-2\int_0^{\omega}\int_0^{\omega}K^{\alpha}_{c_N}(x,y)F^{\alpha}_{c_N}(x,y)dxdy-\int_0^{\omega}\int_0^{\omega}\left(F^{\alpha}_{c_N}(x,y)\right)^2dxdy+O\left(\frac{\omega^2}{c_N}\right).
\end{eqnarray*}
On the other hand, we have
\begin{eqnarray*}
K^{\alpha}_{c_N}(x,y)&=&c_N\frac{\sqrt{xy}}{x^2-y^2}\left(xJ_{\alpha+1}(c_Nx)J_{\alpha}(c_Ny)-yJ_{\alpha+1}(c_Ny)J_{\alpha}(c_Nx)\right)\\&=&\frac{x\sqrt{c_Nx}J_{\alpha+1}(c_Nx)}{x+y}\left(\frac{\sqrt{c_Ny}J_{\alpha+1}(c_Ny)-\sqrt{c_Nx}J_{\alpha+1}(c_Nx)}{x-y}\right)\\&+&\frac{\sqrt{c_Nx}J_{\alpha+1}(c_Nx)}{c_N(x+y)}\left(\frac{(c_Ny)^{3/2}J_{\alpha+1}(c_Ny)-(c_Nx)^{3/2}J_{\alpha+1}(c_Nx)}{x-y}\right),
\end{eqnarray*}
then, from \cite{A.O} and  derivatives formulas associated with Bessel function given in \cite{Watson}, we get
\begin{equation}
\left|K^{\alpha}_{c_N}(x,y)\right|\leq C_{\alpha}\frac{1}{x+y}.
\end{equation}
Further, we have \begin{equation}|F^{\alpha}_{c_N}(x,y)|\leq \frac{2}{\pi}\left[\frac{x+y}{4-(x+y)^2}+\frac{|x-y|}{4-(x-y)^2}\right].\end{equation}
Finally, from the two last inequalities and by a straightforward computations, we get the following estimate

\begin{eqnarray}\label{e10}
\norm{\mathcal{Q}^{\alpha}_{c}}^2_{HS}-\norm{\tilde{\mathcal{Q}}^{\alpha}_{N,\omega}}^2_{HS}&\leq&\frac{C_{\alpha}}{\pi}\int_0^{\omega}\int_0^{\omega}\left[\frac{2}{4-(x+y)^2}+\frac{1}{(x+y)}\frac{2|x-y|}{(4-(x-y)^2)}\right]dxdy+O\left(\frac{\omega^2}{c_N}\right)\\&\leq&\frac{C_{\alpha}}{\pi}\int_0^{\omega}\int_0^{\omega}\left[\frac{2}{4-(x+y)^2}\right]dxdy+\frac{2C_{\alpha}}{\pi}\int_0^{\omega}\int_y^{\omega}\left[\frac{1}{(x+y)}\frac{2(x-y)}{(4-(x-y)^2)}\right]dxdy+O\left(\frac{\omega^2}{c_N}\right)\nonumber\\&\leq&\frac{C_{\alpha}}{2\pi}\int_0^{\omega}\int_0^{\omega}\left[\frac{1}{2-(x+y)}+\frac{1}{2+(x+y)}\right]dxdy\nonumber\\&+&\frac{2C_{\alpha}}{\pi}\int_0^{\omega}\int_y^{\omega}\left[\frac{1}{2(1+y)(2-(x-y))}+\frac{1}{2(1-y)(2+(x-y))}\right]dxdy\nonumber\\&-&\frac{2C_{\alpha}}{\pi}\int_0^{\omega}\int_y^{\omega}\left[\frac{y}{(1-y^2)(x+y)}\right]dxdy+O\left(\frac{\omega^2}{c_N}\right)\nonumber\\&\leq&\frac{C_{\alpha}}{\pi}\left[\omega\ln\left(\frac{1+\omega}{1-\omega}\right)+\ln\left(\frac{4(1-w^2)}{4-\omega^2}\sqrt{\frac{2-\omega}{2+\omega}}\right)+2\frac{\omega(2+\omega)}{1-\omega^2}\ln\left(1+\frac{\omega}{2}\right)\right]+O\left(\frac{\omega^2}{c_N}\right)\nonumber.
\end{eqnarray}
Finally, from \eqref{e8}, \eqref{e9} and \eqref{e10}, we get the desired result.	
	\end{proof}
\section{Application}
Our aim in this section deals with the universal constant appearing in Ingham concentration inequality given in \cite{I}. Let's recall first the Ingham result's
\begin{theorem}[Ingham \cite{I}]
	Let $T>1$ a real fixed number and $N\geq 1$ be an integer, then there exist a constant $C(T)$ depending only of $T$ such that for every complex numbers $a_1,...,a_N$ and $\lambda_1<...<\lambda_N$ $N$ real numbers satisfy $|\lambda_{k+1}-\lambda_{k}|\geq \gamma>0,$ $\gamma >\frac{\pi}{T},$ we have
	\begin{equation*}
	\frac{1}{2T}\int_{-T}^{T}\left|\sum_{k=1}^Na_ke^{i\pi \lambda_kt}\right|^2dt\geq C(T)\sum_{k=1}^N|a_k|^2.
	\end{equation*} 
	In particular, for every $N$ integers $n_1<...<n_N$ and $T>1,$ we have 
	\begin{equation}
	\frac{1}{2T}\int_{-T}^{T}\left|\sum_{k=1}^Na_ke^{i\pi n_kt}\right|^2dt\geq C(T)\sum_{k=1}^N|a_k|^2,\end{equation}
\end{theorem}
which we can rewrite on the following form using the simple change of variable $u=\frac{t}{2N(T)},$
\begin{equation}\label{e12}
\int_{-\frac{T}{2N(T)}}^{\frac{T}{2N(T)}}\left|\sum_{k=1}^Na_ke^{2i\pi N(T)n_kt}\right|^2dt\geq\frac{T}{N(T)} C(T)\sum_{k=1}^N|a_k|^2,
\end{equation}
where $N(T)=[T]+1$. Let consider a finite suitable sequence of integers $n_1<...<n_N$ and $a_1<...<a_N$ be $N$ complex numbers. Let $S_{n_N}=\{(b_k)_{k\in\mathbb{N}}\in\ell^2(\mathbb{R}),\,b_k=0,\,\forall k>N(T)n_N\}$ be the set of index limited sequence over $[|1,...,N(T)n_N|],$ so by taken the function $\varphi_N(t)=\sum_{k=1}^Nb_{N(T)n_k}e^{2i\pi N(T)n_kt},$ we write \eqref{e12} on the following form,
\begin{equation}
 \frac{T}{N(T)}C(T)\leq\frac{\norm{\varphi_N}_{L^2(-\frac{T}{2N(T)},\frac{T}{2N(T)})}}{\norm{\varphi_N}_{L^2(-1/2,1/2)}},
\end{equation}
where $b_{N(T)n_k}=a_k,\,1\leq k\leq N.$ Let $\varDelta_N=\{N(T)n_1,...,N(T)n_N\}$ and  $S_N=\{(b_j)_{j\in\mathbb{N}}\in\ell^2(\mathbb{R}),\,b_j=0,\,\forall j\notin\varDelta_N\}$ be the subspace of $S_{n_N}$ of dimensional $N.$ It's easy to see that the sequence $(b_k\chi_{\varDelta_N}(k))_{k\in\mathbb{N}}\in S_N$ and $\varphi_N$ is their amplitude spectra. So from \eqref{e3}, which we can generalize it using the different details given in general principle section, we have
\begin{equation}
C(T)\leq\frac{N(T)}{T} \tilde{\lambda}_{N}^{N(T)n_N}\left(\frac{T}{2N(T)}\right). 
\end{equation}
The decay rate of the classical eigenvalue $\tilde{\lambda}_{n}^{N}(\omega)$ was given in \cite{B.B.K} equation $(12)$, that is for every $0<\omega<1/2,$ and $\frac{e\pi}{2}N\omega\leq n\leq N-1,$ we have
\begin{equation}
\tilde{\lambda}_{n}^{N}(\omega)\leq A_\omega e^{-(2n+1)\ln\left(\frac{2(n+1)}{e\pi N\omega}\right)},
\end{equation}
where $A_\omega=2\pi^2\left(\frac{1/4-\omega^2}{\cos(\pi\omega)}\right)^2.$ So, with our notations, we take $\omega=\frac{T}{2N(T)}, n=N$ and $N=N(T)n_N,$ then for all $\frac{e\pi}{4}n_NT\leq N\leq N(T)n_N-1$, we get the following inequality
\begin{equation}
C(T)\leq A_T\,\frac{N(T)}{T}e^{-(2N+1)\ln\left(\frac{4(N+1)}{e\pi n_N T}\right)},
\end{equation}
where $A_T=\frac{\pi^2}{8}\left(\frac{1-\frac{T^2}{N^2(T)}}{\cos\left(\frac{\pi}{2}\frac{T}{N(T)}\right)}\right)^2.$ Finally, we have the following result
\begin{equation}
C(T)\leq \frac{\pi^2}{8e^2}\frac{N(T)}{T}\left(\frac{1-\frac{T^2}{N^2(T)}}{\cos\left(\frac{\pi}{2}\frac{T}{N(T)}\right)}\right)^2.
\end{equation}
When $T\to\infty,$ we obtain a better upper bound for Besikovitch norm collected with the lower bound constant given in \cite{J. S}, we get the following result. $$ 0.15\sim\frac{\pi^2}{64}\leq\displaystyle\overline{\lim}(C(T))\leq\frac{2}{e^2}\sim 0.27.$$  


\begin{thebibliography}{99}
	\bibitem{Abreu} 
	Abreu, L. D., Bandeira, A. S. Landau's necessary density conditions for the Hankel transform. Journal of Functional Analysis, (2012), {\bf 262}(4), 1845--1866.
	
	\bibitem{B.N}
	Batir, N. Inequalities for the gamma function. Archiv der Mathematik, (2008), {\bf 91}(6), 554--563.
	
	\bibitem{B.K}
	Bonami, A., Karoui, A. Spectral decay of time and frequency limiting operator. Applied and Computational Harmonic Analysis, (2017), {\bf 42}(1) 42(1), 1--20.
	
	 \bibitem{B.J.K}
    Bonami, A., Jaming, P. and Karoui, A. Non-asymptotic behavior of the spectrum of the sinc-kernel operator and related applications. Journal of Mathematical Physics, (2021), {\bf 62}(3), 033511.
     
     \bibitem{Karoui-Boulsane} Boulsane M, Karoui A. The Finite Hankel Transform Operator: Some Explicit and Local Estimates of the Eigenfunctions and Eigenvalues Decay Rates. J. Four. Anal. Appl. (2018), {\bf 24}, 1554--1578.
     
     \bibitem{B.B.K}
     Boulsane, M., Bourguiba, N., and Karoui, A. Discrete prolate spheroidal wave functions: Further spectral analysis and some related applications. Journal of Scientific Computing, (2020), {\bf 82}(3), 1--19.
     
     \bibitem{B.M}
     Boulsane, M. Approximation in Hankel Sobolev Space by Circular prolate spheroidal series. 
     J. Four. Anal. Appl. (2024).
     
     \bibitem{M.B}
     Boulsane, M. Non-asymptotic behavior and the distribution of the spectrum of the finite Hankel transform operator. Integral Transforms and Special Functions, (2021), {\bf 32}(12), 948--968.
     \bibitem{B. J. S}
     Boulsane, M., Jaming, P. and Souabni, A. Mean convergence of prolate spheroidal series and their extensions. Journal of Functional Analysis, (2019), {\bf 277}(12), 108--295.
     
     	\bibitem{Elbert}
     \'A. Elbert, Some recent results on the zeros of Bessel functions and orthogonal polynomials. Journal of Computational and Applied Mathematics. (2001), {\bf 133}, 65--83.
     
     \bibitem{J. S}
     P. Jaming and C. Saba, From Ingham to Nazarov's inequality: a survey on some trigonometric inequalities.
     	arXiv:2311.17714
     
     \bibitem{H}
     H. Hochstadt, The mean convergence of Fourier–Bessel series, SIAM Rev. (1967), {\bf 9}, 211--218.
     
     \bibitem{H}
     Hogan, J. A., Lakey, J. D. and Lakey, J. D. Duration and bandwidth limiting: prolate functions, sampling, and applications. (2012), Boston, MA: Birkhäuser.
     
     \bibitem{I}
     A. E. Ingham, Some trigonometrical inequalities with applications to the theory of series. Math. Z. (1936), {\bf 41}, 367--379.
     
     	\bibitem{G.D.D}
     Gruenbacher, Don M., and Donald R. Hummels. "A simple algorithm for generating discrete prolate spheroidal sequences." IEEE Transactions on signal processing. (1994), {\bf 42}(11), 3276--3278.
     
     \bibitem{G.P.R}
     Guadalupe, J. J., et al. Mean and weak convergence of Fourier-Bessel series. Journal of mathematical analysis and applications. (1993), {\bf 173}(2), 370--389.
     
     \bibitem{L. P}
     Landau, H.J., Pollak, H.O.: Prolate spheroidal wave functions, Fourier analysis and uncertainty-III. The
     dimension of space of essentially time-and band-limited signals. Bell Syst. Tech. J. (1962), {\bf 41}, 1295--1336.
     
     \bibitem{Ro}
     MacRobert, T. M. "Asymptotic expressions for the Bessel functions and the Fourier-Bessel expansions." Proceedings of the Edinburgh Mathematical Society. (1920), {\bf 39}, 13--20.
	
	\bibitem{A.O}
	Olenko AY . Upper bound on $\sqrt{x}J_\mu(x)$ and its applications. Integral Transforms and Special
	Functions. (2006), {\bf 17}, 455--467.
	
	\bibitem{R.B}
	R.B.Paris. An inequality for the Bessel function $J_{\nu}(\nu x).$ SIAM J. Math. Anal. (1984), {\bf 15}(1), 203--205.
	
	\bibitem{Slepian0} 
	Slepian, D. Prolate spheroidal wave functions, Fourier analysis, and uncertainty—V: The discrete case. Bell System Technical Journal,  (1978), {\bf 57}(5), 1371--1430.
	
	\bibitem{Slepian1}
	Slepian D, and H. O. Pollak. "Prolate spheroidal wave functions, Fourier analysis and uncertainty—I." Bell System Technical Journal. (1961), {\bf 40}(1), 43--63.
	
	\bibitem{Slepian3} Slepian D. Prolate spheroidal wave functions, Fourier analysis and uncertainty--IV: Extensions to many dimensions; generalized
	prolate spheroidal functions. Bell System Tech. J. (1964), {\bf 43}, 3009--3057.

	\bibitem{Slepian2}
	Slepian, D. Some asymptotic expansions for prolate spheroidal wave functions. Stud. Appl.Math. (1965), {\bf 44}(4),
	99--140. 
	
	\bibitem{T.R.J}
	Tzschoppe, Roman, and Johannes B. Huber. "Causal discrete‐time system approximation of non‐bandlimited continuous‐time systems by means of discrete prolate spheroidal wave functions." European transactions on telecommunications. (2009), {\bf 20}(6), 604--616.
		
	\bibitem{V}
	Varona, J. L. Fourier series of functions whose Hankel transform is supported on [0,1]. Constructive Approximation, (1994), {\bf 10}(1), 65--75.
	
	\bibitem{Szego} 
	Szeg, G. Orthogonal polynomials (Vol. 23). American Mathematical Soc..(1939).
	
	\bibitem{W.Z}
	L. L. Wang, J. Zhang, A new generalization of the PSWFs with applications to spectral approximations on quasi-uniform grids.
	J. Appl. Comput. Harmon. Anal. (2010), {\bf 29}, 303--329.
	
	\bibitem{W}
	L. L. Wang. Analysis of Spectral Approximations using Prolate Spheroidal Wave Functions. J.
	Math. Comput. (2010), {\bf 270}, 807--827.
	
	\bibitem{Watson}
	 G.N. Watson . A treatise on the theory of Bessel functions. Second edition, Cambridge University Press. 1966.
	\end{thebibliography}
\end{document}